\documentclass[12pt]{amsart}

\usepackage{amssymb}
\usepackage{amsthm}

\usepackage[english]{babel}
\usepackage[utf8]{inputenc}
\usepackage[T1]{fontenc}

\usepackage{enumerate}
\usepackage{enumitem}

\usepackage{amsmath}
\usepackage{mathtools}
\usepackage{mathabx}

\usepackage[colorlinks=true,hyperindex,linktocpage=true]{hyperref}
\usepackage{fullpage}

\theoremstyle{plain}
\newtheorem{Satz}{Theorem}[section]
\newtheorem{Lemma}[Satz]{Lemma}
\newtheorem{Korollar}[Satz]{Corollary}
\newtheorem{Proposition}[Satz]{Proposition}

\theoremstyle{definition}
\newtheorem{Definition}[Satz]{Definition}
\newtheorem{Beispiel}[Satz]{Example}
\theoremstyle{remark}
\newtheorem{Bemerkung}[Satz]{Remark}

\newcommand{\Z}{\mathbb{Z}}
\newcommand{\R}{\mathbb{R}}
\newcommand{\C}{\mathbb{C}}
\newcommand{\Q}{\mathbb{Q}}
\newcommand{\PP}{\mathbb{P}}
\newcommand{\HH}{\mathbb{H}}

\DeclareMathOperator{\im}{Im}
\DeclareMathOperator{\SL}{SL_2}
\DeclareMathOperator{\SO}{SO}
\DeclareMathOperator{\OO}{O}
\DeclareMathOperator{\GL}{GL}
\DeclareMathOperator{\SLZ}{SL_2(\Z)}

\DeclareMathOperator{\SLR}{SL_2(\R)}
\DeclareMathOperator{\MPR}{Mp_2(\R)}
\DeclareMathOperator{\MPZ}{Mp_2(\Z)}

\DeclareMathOperator{\Spin}{Spin}
\DeclareMathOperator{\sign}{sign}
\DeclareMathOperator{\ef}{\mathfrak{e}}
\DeclareMathOperator{\eb}{\mathbf{e}}
\newcommand{\Ma}[4]{\begin{pmatrix} #1 & #2 \\ #3 & #4 \end{pmatrix}}
\newcommand{\sMa}[4]{\begin{psmallmatrix} #1 & #2 \\ #3 & #4 \end{psmallmatrix}}

\usepackage[style=alphabetic]{biblatex}

\begin{document}

\title{Invariants for the Weil representation and Modular Units for Orthogonal Groups of Signature $(2,2)$}
\author{Patrick Bieker}

\address{Bielefeld University, Faculty of Mathematics, Postfach 100131, 33501 Bielefeld, Germany}
\email{pbieker@math.uni-bielefeld.de}

\begin{abstract}
	We show that the space of invariants for the Weil representation for discriminant groups which contain self-dual isotropic subgroups is spanned by the characteristic functions of the self-dual isotropic subgroups. 
	As an application, we construct modular units  
	for certain orthogonal groups in signature $(2,2)$ using Borcherds products.
\end{abstract}

\maketitle

\section{Introduction}

In \cite{Bor98}, Borcherds defined a lifting of (weakly holomorphic) vector valued modular forms for the Weil representation attached to an even lattice $L$ of signature $(2,n)$ to modular forms for a subgroup of the orthogonal group of $L$.
Moreover, the divisor of the resulting modular form is known, and these modular forms have infinite product expansions similar to the Dedekind $\eta$-function
\[ \eta(\tau) = \exp ( \pi i \tau /12 ) \prod_{n = 1}^\infty(1 - \exp(2\pi i \tau)^n)\]
at the cusps and are therefore called \emph{Borcherds products}.

We restrict ourselves to holomorphic modular forms for the Weil representation as inputs to the lift. In this case, the corresponding Borcherds products have divisors supported on the boundary.
We call meromorphic modular forms with divisor supported on the boundary \emph{modular units} in reference to the work of Kubert and Lang \cite{KuLa},
who studied modular functions of this form on modular curves for some congruence subgroups.

The Weil representation $\varrho_{L'/L}$ is a representation of $\SLZ$ (or in general of the metaplectic group $\MPZ$, the double cover of the special linear group) on the group ring  \( \C[L'/L] \) associated with the discriminant group \(L'/L\) of \(L\), where we denote by $L'$ the dual lattice of $L$. 
The holomorphic modular forms of weight 0 for the Weil representation 
are invariant vectors in $\C[L'/L]$.
For a subgroup \( H \subset L'/L \) we denote its characterisitic function in \( \C[L'/L] \) by 
\[ v^H = \sum_{\gamma \in H} \ef_\gamma,\]
where $\ef_\gamma$ for $\gamma \in H$ are the standard basis vectors in $\C[L'/L]$. 
\begin{Satz}[cf. Theorem \ref{IsoSubIn}]
	\label{Satz1}
	Assume that  $L'/L$ possesses self-dual isotropic subgroups. Then $\C[L'/L]^{\SLZ}$ is spanned by the characteristic functions $v^H$ of self-dual isotropic subgroups $H$ of $L'/L$.
\end{Satz}
This result is mentioned but not shown for example in \cite[Theorem 1]{Skor08} and \cite{EhlSkor}. We give a proof by translating a similar result in the language of codes \cite[Theorem 5.5.7]{NRS} to our setup.

In the general case, that is, for discriminant forms that not necessarily contain self-dual isotropic subgroups, \cite[Theorem 3.3]{EhlSkor} shows that the space of invariants can be defined over $\Z$ and \cite[Theorems 5.4 and 5.6]{Zemel2021} gives dimension formulas for the space of invariants in many cases.
For arbitrary discriminant forms \cite{Mueller22} gives a generating set of the space of invariants, which generalises our result, proving a conjecture of Scheithauer.

We apply this description of the invariants explicitly to the orthogonal sum of two rescaled hyperbolic planes.
A hyperbolic plane $U$ is a lattice that is isomorphic to the lattice $\Z^2$ with quadratic form $Q((a,b)) = ab$.
For a positive integer $N$ we denote by $U(N)$ a hyperbolic plane with quadratic form rescaled by $N$.
For a divisor $N'$ of $N$ we consider the lattice $ L = L_{N,N'} = U(N) \oplus U(N')$, which is a lattice of signature \( (2,2) \).
We denote its associated Weil representation by $\varrho_{N,N'}$. 
In order to simplify the presentation here, we stick to the case \(N' = 1\) for now. 
The self-dual subgroups of \(L'/L = (\Z/N\Z)^2 \) are of the form
\[ H_d = \left\{ (x,y) \in (\Z/N\Z)^2 \colon d|x, \frac{N}{d} | y \right\} \]
for divisors \(d\) of \(N\). As a consequence, we give a new proof of  \cite[Lemma 3.2]{Ye2021}:
\begin{Korollar}[cf. Corollary \ref{corBasis}]
	The vectors \( \{ v^{H_d} \colon d|N \} \) form a basis of the space of holomorphic modular forms of weight 0 with respect to \( \varrho_{N,1} \).
\end{Korollar}

It is well-known that modular forms on orthogonal groups of signature \( (2,2) \) can be realised as modular forms for discrete subgroups of \( \SLR \times \SLR \)  
on \( \HH \times \HH \), where $\HH$ stands for the upper half-plane. 
We obtain the following explicit description of Borcherds products for our lattice \(L\).
\begin{Satz}[cf. Corollary \ref{corN1}]
	Let \(F = \sum_{d|N} \alpha_d v^{H_d} \) be a holomorphic modular form of weight 0 with respect to the Weil representation \( \varrho_{N,1} \) with integer coefficients \(\alpha_d \in \Z\). The Borcherds lift of \(F\), denoted by $\Psi \colon \HH^2 \to \C$, is a modular form of weight \( \frac 1 2 \sum_{d|N} \alpha_d \) with a multiplier system \( \chi\) with respect to \(\Gamma_0(N)^2\). Moreover, the divisor of \( \Psi \) vanishes on \( \HH^2 \) and \(\Psi\) has a product expansion
	\[ \Psi(\tau_1, \tau_2) = C \prod_{d|N} \eta(d\tau_1)^{\alpha_d} \eta(d\tau_2)^{\alpha_d}, \]
	for some non-zero constant \( C \in \C^\times \).
\end{Satz}

The general case (\(N' \neq 1 \)) yields similar results. However, there is one main difference to the \(N' = 1\) case, namely, the characteristic functions of self-dual isotropic subgroups are not linearly independent in general. The lift of linear relations among the characteristic functions gives rise to identities between the corresponding \(\eta\)-quotients.
In the case $N=N'=p$ for some prime number $p$, there is a single relation between the characteristic functions of self-dual isotropic subgroups of $D_{p,p}$. Lifting this relation yields the following result.  

\begin{Satz}[cf. Corollary \ref{thmEta}]
	Let \(p\) be a prime number. We have
	\[ \prod_{a = 1}^{p-1} \eta\left( \tau + \frac{a}{p} \right) = \exp\left(\frac{ 2\pi i(p-1)}{48}\right) \frac{\eta(p\tau)^{p+1}}{\eta(\tau) \eta(p^2 \tau)}.\]
\end{Satz} 

This identity is well-known for \(p=2\), see for example \cite{Eta}, but for $p \geq 3$ the result seems to be new.
Identities between eta quotients are also studied in particular in relation to theta functions and theta derivatives. For example, \cite{Oliver2013}, \cite{Matsuda2016} and  \cite{Zemel2019} give various constructions of such identities. \cite{Cooper2017} treats various types of relations between theta functions.

This paper is organised as follows. In Section 2, we introduce the objects that we need to formulate Borcherds' theorem, namely the Weil representation, vector valued modular forms, and modular forms for orthogonal groups.
In Section 3, we study invariants for the Weil representation and prove Theorem \ref{Satz1}.
In Section 4, we explicitly determine self-dual isotropic subgroups in the discriminant group of the lattice $L_{N,N'}$, which we use to construct Borcherds products in Section 5.

\section{Preliminaries}

We briefly review the notions necessary to formulate Borcherds' theorem.
We denote the ordinary upper half plane $\{ \tau \in \C \colon \im \tau > 0 \}$ by $\HH$, and use the usual conventions \( \eb(z) = \exp(2 \pi i z) \) for \(z \in \C\) and \(q = \eb(\tau)\) for \( \tau \in \HH\), and similarly \(q_1 = \eb(\tau_1)\) and \(q_2 = \eb(\tau_2) \) for $(\tau_1, \tau_2) \in \HH \times \HH$.

\subsection{Weil Representations and Vector Valued Modular Forms}

Vector valued modular forms for the Weil representation of the metaplectic group $\MPZ$ play an important role in the theory of Borcherds products as they serve as input to the theta lift.
We follow \cite{JHB} and \cite{Scheit}.

Recall that $\MPR$ is the \emph{metaplectic group}, i.e. the double cover of $\SLR$ realised by the two choices of holomorphic square roots of $\tau \mapsto (c\tau + d)$ for $\sMa a b c d \in \SLR$.
More formally, elements of $\MPR$ are pairs $(M, \phi )$, where $M = \sMa a b c d$ is an element of $\SLR$ and $\phi \colon \HH \to \C$ is a holomorphic function such that $\phi(\tau)^2 = c\tau + d$.
Let $\MPZ$ be the inverse image of $\SLZ$ under the covering map $\MPR \to \SLR$.
It is well-known that $\MPZ$ is generated by the two elements
$$ T = \left( \Ma 1 1 0 1 , 1 \right) \qquad \text{ and } \qquad S = \left( \Ma 0 {-1} 1 0, \sqrt{\tau} \right).$$

A \emph{discriminant form} (or finite quadratic module) is a finite abelian group $D$ together with a non-degenerate quadratic form $Q \colon D \to \Q/\Z$. We denote by $B \colon D \times D \to \Q/\Z$ its associated bilinear form. 
The \emph{level} of a discriminant form is the smallest positive integer $N$ such that $N Q(\gamma) = 0 $  in $\Q/\Z$ for all $\gamma \in D$.
We denote by $\sign (D) \in \Z/8\Z$ the signature of $D$.
We denote the group ring associated with $D$ by $\C[D]$ and use $\mathfrak{e_\gamma}$ for $\gamma \in D$ as its standard basis.
\begin{Definition} 
	The \emph{Weil representation} $\varrho_D$ attached to a discriminant form $(D, Q)$ is defined as the representation
	\[\varrho_D \colon \MPZ \to \GL( \C[D]), \]
	such that for all $\gamma \in D$:
	\begin{align*}
		\varrho_D(T) \mathfrak{e}_\gamma &= \mathbf{e} (Q(\gamma)) \mathfrak{e}_\gamma \\
		\varrho_D(S) \mathfrak{e}_\gamma &= \frac{\eb(-{\sign(D)}/8)}{\sqrt{|D|}} \sum_{\beta \in D} \mathbf{e}(-B(\beta, \gamma)) \ef_\beta .
	\end{align*}
\end{Definition}
It is well known that if $\sign(D)$ is even, the representation $\varrho_D$ factors through $\SLZ$ and even $\SL(\Z/N\Z)$, where \(N\) is the level of \(D\).

We use the Weil representation to define vector valued modular forms in the following way.
For $k \in \frac 1 2 \Z$, a $\C[D]$-valued function $F$ on $\HH$, and $(M, \phi) \in \MPR$ we define the Petersson slash operator by
$$ ( F|_{k,D} (M, \phi))(\tau) = \phi^{-2k}(\tau) \varrho_D^{-1}(M, \phi) F(M\tau).$$

\begin{Definition}
	\index{modular form!for the Weil representation}
	A \emph{weakly holomorphic modular form} of weight $k$ with respect to $\varrho_D$ (and $\MPZ$) is a holomorphic function $F \colon \HH \to \C[D]$ such that
	\begin{enumerate}
		\item $F|_{k,D} (M, \phi) = F$  for all $(M, \phi) \in \MPZ$ and
		\item $F$ is meromorphic at $\infty$, that is, $F$ has a Fourier expansion of the form 
		$$F(\tau)= \sum_{\gamma \in D} \sum_{\substack{n \in \Z + Q(\gamma) \\ n >> -\infty}} c_\gamma(n) q^n \ef_\gamma.$$
	\end{enumerate}
	We say that $F$ is a \emph{holomorphic modular form} if $F$ is in addition holomorphic in $\infty$, that is if all Fourier coefficients with negative indices vanish.
	We denote the space of holomorphic modular forms of weight $k$ with respect to $\varrho_D$ and $\MPZ$ by $\mathcal{M}_{k,D}$.
\end{Definition}

Later we will be interested, in particular, in the space of holomorphic modular forms of weight 0, as we use them as inputs for Borcherds' lift.
The component functions of such holomorphic weight 0  forms  are elliptic modular forms for $\Gamma(N)$, so that in particular they are constant.
Thus, holomorphic modular forms of weight 0 with respect to $\varrho_D$ are constant vectors that are invariant under the operation of the Weil representation of $\SLZ$. 
We denote the space of invariant vectors under the Weil representation by $\C[D]^{\SLZ}$. 

\subsection{Modular Forms for Orthogonal Groups}
\label{secModOrth}
We recall the notions of orthogonal modular forms in the case of signature $(2,n)$ and focus in particular on the case $n = 2$.
Orthogonal modular forms for discriminant kernels $\Gamma_L$ of lattices $L$ arise as the result of the Borcherds lift.
We follow \cite{JHB} and \cite{123}.

Let $(V, Q)$ be a non-degenerate quadratic space over $\Q$. We denote its associated bilinear form by $B$. We write $V_\R$ and $V_\C$ for $V \otimes_\Q \R$ and $V \otimes_\Q \C$ respectively, and assume that $V_\R$ has signature $(2,n)$. 
We denote by $\OO^+(V_\R)$ the subgroup of $\OO(V_\R)$ of elements whose spinor norm equals their determinant. Then $\SO^+(V_\R) = \OO^+(V_\R) \cap \SO(V_\R)$ is the identity component of $\OO(V_\R)$.

It is well-known that in this case $\OO(V_\R)/K$, for a maximal compact subgroup $K$ of $\OO(V_\R)$, is a hermitian symmetric domain. 
We can realise $\OO(V_\R)/K$ as a tube domain in the following way. We can split a hyperbolic plane $U_\R = U \otimes_\Z \R$ from $V_\R$ as $V_\R = U_\R \oplus V_1$. Then $V_1$ is a real quadratic space of signature $(1,n-1)$. We denote by $C = \{ v \in V_1 \colon Q(v) > 0\}$ the set of positive norm vectors in $V_1$. It has two components, and we fix one component $C^+$ in $C$ and set $\mathcal{K}' = V_1 + iC$ and $\mathcal{H}' = V_1 + iC^+$. It is well-known that $\mathcal{H}'$ is a realisation of the hermitian symmetric domain $\OO(V_\R)/K$. 

It is also convenient to realise $\OO(V_\R)/K$ inside the projective space $\PP(V_\C)$ as follows. Let $\mathcal{N} = \{ [Z] \in \PP(V_\C) \colon Q(Z) = 0\}$ be the zero quadric in $ \PP(V_\C)$. We set $\mathcal{K} = \{ [Z] \in \mathcal{N} \colon B(Z, \bar{Z}) > 0\}$, it is an analytically open subset of $\mathcal{N}$.
It consists of two components that are preserved by $\OO^+(V_\R)$. 
In order to define a map between the two realisations, we choose a basis $\kappa, \kappa'$ of $U$ such that $\kappa$ is isotropic and $B(\kappa, \kappa') = 1$. Then the map
$$ Z \mapsto [Z + \kappa' -(Q(Z) + Q(\kappa')) \kappa] $$
biholomorphically identifies $\mathcal{K}'$ with $\mathcal{K}$ and $\mathcal{H}'$ with a component of $\mathcal{K}$, which we denote by $\mathcal{H}$.  

Moreover, it is well-known that in signature $(2,2)$ the domain $\mathcal{H}$ can be identified biholomorphically with $\HH \times \HH$, the product of two upper half planes. For this, we choose a basis $\tilde{\kappa}$, $\tilde{\kappa}'$ of isotropic vectors of $V_1$ such that $B(\tilde \kappa, \tilde \kappa') = 1$. Then $C = \{ (a,b) \colon a,b > 0\} \cup \{ (a,b) \colon a,b < 0\}.$ We choose the component with positive entries for the definition of $\mathcal{H}'$ and obtain
$$ \mathcal{H}' = \{ (\tau_1, \tau_2) \in V_1 \otimes \C \cong \C ^2 \colon \im(\tau_1), \im(\tau_2) > 0\} \cong \HH \times \HH.$$
In addition, in signature $(2,2)$ the double cover $\Spin(V_\R)$ of $\SO^+(V_\R)$ can be identified with $\SLR^2 = \SLR \times \SLR$ and the action of $\SLR \times \SLR$ on $\HH \times \HH$ under these identifications is given by componentwise Möbius transformations.

Let $L$ be an even lattice of signature $(2,n)$. We set $V = L \otimes_\Z \Q$. Let $\Gamma \subset \OO^+(V)$ be a subgroup commensurable with $\SO^+(L)$. 
It is well-known that the quotient $\Gamma \backslash \mathcal{H}$ is a quasi-projective variety which is compact if and only if $V$ is anisotropic.

We briefly sketch how to compactify the quotient in the case where $V$ contains a two-dimensional isotropic subspace using the Satake-Baily-Borel compactification.
To this end, we consider the boundary of $\mathcal{H}$ in $\mathcal{N}$. 
For a real isotropic vector $Z \in V_\R$, the point $[v] \in \PP(V_\C)$ is in the boundary of $\mathcal{H}$ and we call $\{[v]\}$ a zero-dimensional boundary component or special boundary point. 
For a two-dimensional isotropic subspace $F$ of $V$, we call the set of all non-special boundary points which can be represented by elements of $F \otimes \C$ the one-dimensional boundary component associated with $F$. 
Then the boundary of $\mathcal{H}$ is the disjoint union of the boundary components and the set of boundary components is in bijection to the set of non-zero isotropic subspaces of $V_\R$. 
We call a boundary component \emph{rational} if it is defined over $\Q$. 
We refer to rational zero-dimensional boundary components as \emph{cusps}.
The one-dimensional boundary components are biholomorphic to $\HH$.
We denote the union of $\mathcal{H}$ with all rational boundary components by $\mathcal{H}^*$.
We equip $\mathcal{H}^*$ with the Satake-Bailey-Borel topology.
Then the quotient 
$ X(\Gamma) = \Gamma \backslash \mathcal{H^*}$
is compact and $X(\Gamma)$ carries the structure of a projective variety which contains $\Gamma \backslash \mathcal{H}$ as an open subvariety.
In the case of signature $(2,2)$ the rational cusps of $\mathcal{H} \cong \HH \times \HH$ are points in $\PP^1(\Q)\times \PP^1(\Q)$ and the one-dimensional rational boundary components are of the form $\{ a\} \times \HH$ or $\HH \times \{ a\}$ for $a \in \PP^1(\Q)$.

We now define orthogonal modular forms in our setting in signature $(2,2)$. For a general definition, we refer to \cite{JHB} and \cite{123}. 
For $M_1, M_2 \in \SLR$ with $M_1 = \sMa {a_1} {b_1} {c_1} {d_1}$ and $M_2 = \sMa {a_2} {b_2} {c_2} {d_2}$ we denote the \emph{automorphy factor}
$$j\left( (M_1, M_2) , (\tau_1, \tau_2) \right) = (c_1 \tau_1 + d_1) (c_2 \tau_2 + d_2).$$
We fix a holomorphic logarithm $\mathrm{Log}(j\left( (M_1, M_2) , (\tau_1, \tau_2) \right)) $ of the automorphy factor and for $r \in \Q$ we set
$$ j\left( (M_1, M_2) , (\tau_1, \tau_2) \right)^r = \exp(r \cdot \mathrm{Log}(j\left( (M_1, M_2) , (\tau_1, \tau_2) \right)) )$$
for $M_1, M_2 \in \SLR$ and $\tau_1, \tau_2 \in \HH$. Then $\frac{j(\sigma_1 \sigma_2, \tau)^r}{j(\sigma_1, \sigma_2 \tau)^r j(\sigma_2, \tau)^r}$ is a root of unity of order bounded by the denominator of $r$. A \emph{multiplier system} of weight $r$ with respect to $\Gamma$ is a map
$ \chi \colon \Gamma \to \lbrace t \in \C \colon |t| = 1 \rbrace$
sastisfying
$$ \chi(\sigma_1 \sigma_2) = \frac{j(\sigma_1 \sigma_2, \tau)^r}{j(\sigma_1, \sigma_2 \tau)^r j(\sigma_2, \tau)^r}  \chi(\sigma_1) \chi(\sigma_2)$$
for all $\sigma_1, \sigma_2 \in \Gamma$ and $\tau \in \mathcal H$. 
If $r \in \Z$, then the multiplier system $\chi$ is a character.

\begin{Definition}
	Let $\Gamma \subset \SLR \times \SLR$ be commensurable with $\SO^+(L)$.
	A (meromorphic) modular form  of weight $r$ and multiplier system $\chi$ with respect to $\Gamma$ is a meromorphic function $f$ on $\HH \times \HH$ with the transformation property
	$$f(M_1 \tau_1, M_2 \tau_2) = j\left( (M_1, M_2) , (\tau_1, \tau_2) \right)^{r}  \chi(M_1, M_2) f(\tau_1, \tau_2)$$
	for $M_1, M_2 \in \Gamma$.
\end{Definition}

For a lattice $L$ we denote by $L'$ its \emph{dual lattice}, i.e. $L' = \{ x \in L \otimes \Q \colon B(x,y) \in \Z \text{ for all } y \in L\}$. The quotient $L'/L$ is a discriminant form, called the \emph{discriminant group} of $L$. 
We are  interested, in particular, in modular forms for the \emph{discriminant kernel} $\Gamma = \Gamma_{L}$ of $L$, that is, the subgroup of $\SO^+(L)$ that acts trivially on the discriminant group $L'/L$. It is a subgroup of finite index in $\SO^+(L)$.

\subsection{Borcherds' Lift for $L_{N,N'}$}
\label{secLift}

Recall that we defined for a positive integer $N$ and a divisor $N'|N$ the lattice $L = L_{N,N'} = U(N) \oplus U(N')$ as the orthogonal sum of two rescaled hyperbolic planes over $\Z$.
We state a special case of Borcherds' theorem \cite[Theorem 13.3]{Bor98} for $L_{N,N'}$. We present only what is strictly necessary in the following and refer to \cite{Bor98} and \cite{JHB} for the general statement, further details, and proofs. \cite{Ye2021} considers Borcherds products for the lattice $L_{N,1}$ in a different setting.

More precisely, as abelian groups $L \cong \Z^4$ and the quadratic form is given by $Q(w,x,y,z) = Nwx + N'yz$. In particular, $L$ has signature $(2,2)$.
Let $V = L \otimes_\Z \Q$ be the associated rational vector space, and we denote by $V_\R$ and $V_\C$ be the corresponding real and complex vector spaces.
The discriminant group of $L$ is then given by $L'/L \cong (\Z/N\Z)^2 \oplus (\Z/N'\Z)^2$.
We denote the Weil representation associated with $L_{N,N'}$ 
by $\rho_{N,N'}$.

Note that the lattice $L$ can be realised as a lattice of certain integral $(2 \times 2)$-matrices via
\begin{align}
	L & \rightarrow M_2(\Z) \nonumber \\
	(w,x,y,z) & \mapsto \begin{pmatrix} z & x \\ -Nw & N'y \end{pmatrix}
	\label{mapIden}
\end{align}
with its quadratic form given by the determinant. This realisation then identifies $V$ with the vector space of rational $2 \times 2$-matrices. Under these identifications, the action of $\mathrm{Spin}(V_\R) \cong \SLR \times \SLR$ on $V_\R \cong M_2(\R)$ is then given by $((M_1, M_2), X) \mapsto M_1 X M_2^{-1}$ for all $X \in V_\R$ and $(M_1, M_2) \in \SLR \times \SLR$.

\begin{Proposition}
	The above identification $\mathrm{Spin}(V_\R)  \cong \SLR \times \SLR$ induces an isomorphism
	\begin{equation}
		\Gamma_L \cong \left\{ \left(
	\begin{pmatrix}
		a_1 & b_1 \\
		NN'c_1 & d_1
	\end{pmatrix},
	\begin{pmatrix}
		a_2 & N'b_2 \\
		Nc_2 & d_2
	\end{pmatrix}
	\right) \in \SLZ^2
	\colon 
	\begin{array}{l}
		a_i, b_i,c_i, d_i \in \Z, \ i = 1,2 \\
		a_1 a_2  \equiv 1 \textit{\emph{ mod }} N, \\
		a_1 d_2 \equiv 1 \textit{\emph{ mod }} N'
	\end{array}
	\right\}/ \{\pm 1\}.
	\label{eqnGammaL}
	\end{equation}
\end{Proposition}
Note that the representation of $\Gamma_L$ really depends on the choice of realization of $L$. Different choices of realisations of $L$ inside $M_2(\R)$ yield realisations of $\Gamma_L$ that are conjugate to the one given in the proposition.
\begin{proof}
	Let $(M_1, M_2) = (\begin{psmallmatrix}
		a_1 & b_1 \\
		c_1 & d_1
	\end{psmallmatrix}, \begin{psmallmatrix}
		a_2 & b_2 \\
		c_2 & d_2
	\end{psmallmatrix} ) \in \SLR \times \SLR$. Then $(M_1, M_2)$ operates trivially on $L'/L$ if and only if it changes any of the standard basis vectors of $L'$ by an element of $L$. We evaluate the operations on the standard basis vector of $L'$ and obtain e.g. for $X =  \sMa {\frac{1}{N'}} 0 0 0$ that
	$ M_1XM_2^{-1} = \frac{1}{N'} \begin{psmallmatrix}
			{a_1 d_2} & {-a_1 b_2} \\
			c_1 d_2 & -c_1 b_2
	\end{psmallmatrix}.$
	Hence, we get the conditions $a_1d_2 \equiv 1 \  \mathrm{ mod } \ N'$, $a_1b_2 \in \Z$, $c_1d_2 \in NN'\Z$, $c_1b_2 \in (N')^2 \Z$, and similar conditions for the other entries. Together with $\det M_2 = 1$ these conditions yield 
	$ a_1^2 = (a_1 a_2) (a_1 d_2 ) - (a_1 c_2)(a_1 b_2) \in \Z.$
	Similarly, we get 
	$b_1^2 \in \Z, \ c_1^2 \in (NN')^2 \Z, \ d_1^2 \in \Z,  a_2^2 \in \Z, \ b_2^2 \in (N')^2 \Z, \ c_2^2 \in N^2 \Z,$ and $d_2^2 \in \Z.$
	In particular, we can write $a_1 = \sqrt{\mu} \bar{a}_1$ for some square-free positive integer $\mu$ and $\bar{a}_1 \in \Z$. Together with  conditions of the form $a_1 d_2 \in \Z$ we see that 
	$d_2 = \sqrt{\mu} \bar{d}_2$. In a similar fashion we see that all entries of $M_1$ and $M_2$ have to be a product of $\sqrt{\mu}$ and an integer. Hence, $\mu$ divides $\det M_1$, so $\mu = 1$. In other words, all entries are integral and of the form of (\ref{eqnGammaL}).
	
	In order to see that the conditions of (\ref{eqnGammaL}) are already sufficient, we note that together with $\det M_1 = \det M_2 = 1$ the condition $a_1 a_2 \equiv 1 \ \mathrm{mod} \ N$ implies that $d_1 d_2 \equiv \ \mathrm{mod} \ N$, and  $a_1 d_2 \equiv 1 \ \mathrm{mod} \ N'$ implies that $d_1 a_2 \equiv 1 \ \mathrm{mod} \ N'$.
\end{proof}

We choose a primitive isotropic vector $\kappa \in L$, as well as $\kappa' \in L'$ such that $\kappa$ and $N\kappa'$ form a basis for the first factor of $L$ and $B(\kappa, \kappa') = 1$ .
We set $K = L \cap \kappa^\bot \cap (\kappa')^\bot$. In particular, $K \cong U(N')$ and  $K$ has signature $(1, 1)$. 
For $v \in V_\R$ we write $v_K$ for its orthogonal projection to $K \otimes \R$. 
We consider the sub-lattice
$L_0' = \left\{ \lambda \in L' \colon \ B(\lambda, \kappa) \equiv 0  \mod N \right\}.$
Then the orthogonal projection 
induces a surjection $p \colon L_0'/L \to K'/K$. Note that in general one has to slightly modify the orthogonal projection to get the desired projection $p$. In our case, the two definitions coincide. 

We explicitly choose $\kappa$ to be the second standard basis vector $e_2$ of $L$ (corresponding to the matrix $\sMa{0}{1}{0}{0}$ under the identification (\ref{mapIden})) and $\kappa' = \frac{1}{N} e_1$ (corresponding to $\sMa{0}{0}{1}{0}$). Then $L_0 = (N\Z) \oplus \Z^3$ and the projection $L_0'/L \cong 0 \oplus (\Z/N\Z) \oplus (\Z/N'\Z)^2 \to  K'/K$ is the projection onto $(\Z/N'\Z)^2$.

As the next step, we calculate the Weyl vector associated with $\kappa$ and a holomorphic modular form $F$ of weight $0$ for $\varrho_{N,N'}$. We do not give the general definition of Weyl vectors, but refer to \cite[Section 10]{Bor98} and only use an explicit formula of \cite[Theorem 10.4]{Bor98} to calculate the Weyl vector in our case.
Let 
$ F = \sum_{\lambda \in L'/L} c_\lambda \mathfrak{e_\lambda}.$
We set 
$ F_K = \sum_{\gamma \in K'/K} d_\gamma \mathfrak{e}_\gamma \in \C[K'/K]$ with components
$ d_{\gamma} = \sum_{\lambda \in L_0'/L, \ p(\lambda) = \gamma} c_\lambda$.
We choose a primitive isotropic vector $\tilde{\kappa} \in K$, as well as $\tilde{\kappa}' \in K'$ such that $( \tilde\kappa, \tilde\kappa') = 1$, and consider the sublattice
$K_0' = \left\{ \gamma \in K' \colon \ (\gamma, \tilde{\kappa}) \equiv 0  \mod N' \right\}.$ Explicitly, we take $\tilde{\kappa} = e_4$ (corresponding to $\sMa{1}{0}{0}{0}$) and $\tilde{\kappa}' = \frac{1}{N'} e_3$  (corresponding to $\sMa{0}{0}{0}{1}$). Then $K_0'/K \cong 0 \oplus (\Z/N'\Z)$.
\begin{Proposition}
	\label{WeylVect}
	The \emph{Weyl vector} as defined in \cite[Section 10]{Bor98} is the vector $\rho(K, F_K) = \rho_{\tilde{\kappa}'} \tilde{\kappa}' + \rho_{\tilde{\kappa}} \tilde{\kappa} \in K \otimes \R$ with
	$$ \rho_{\tilde{\kappa}'} = \frac 1 {24} \sum_{\beta_2 \in \Z/N'\Z} \sum_{\beta_1 \in \Z/N\Z} c_{(0,\beta_1,0,\beta_2)}\quad \text{ and } \quad \rho_{\tilde \kappa} = \frac 1 {24N'} \sum_{\beta_2 \in \Z/N'\Z} \sum_{\beta_1 \in \Z/N\Z} c_{(0,\beta_1,\beta_2,0)} .$$
\end{Proposition}

\begin{proof} 
	Note that $K$ has rank 2 and thus the Weyl vector only has two components.
	By \cite[Theorem 10.4]{Bor98} (note that the lattice $K$ in Borcherds' notation is the zero lattice here), we have $\rho_{\tilde{\kappa}'} = \frac{1}{24} \sum_{\gamma \in K_0'/K} d_{\gamma}$. This gives the formula for $\rho_{\tilde{\kappa}'}$.
	
	In order to compute $\rho_{\tilde \kappa }$, we use that the Weyl vector is independent of the choice of $\tilde \kappa$ and $\tilde \kappa'$. 
	Swapping the roles of $\tilde{\kappa}$ and $\tilde{\kappa}'$ for the calculation, more precisely, choosing $\tilde{\kappa}_1 = e_3 = N'\tilde{\kappa}'$ and $\tilde{\kappa}_1' = \frac{1}{N'} e_4 = \tilde{\kappa}/N'$ yields
	$ \rho_{\tilde{\kappa}} = \frac{1}{N'}\rho_{\tilde \kappa'_1}$ as claimed above.
\end{proof}

We can now formulate Borcherds' Theorem in this setting. 
\begin{Satz}[compare {\cite[Theorem 13.3]{Bor98}}]
	\label{corBP2g}
	Let $F = \sum_{\gamma \in L'/L} c_\gamma \mathfrak{e}_\gamma$ be a holomorphic modular form of weight 0 with respect to $\varrho_{N,N'}$ with integer coefficients. Then there exists a modular form $\Psi \colon \HH \times \HH \to \C$ of weight $c_0/2$ with respect to $\Gamma_L$ and some multiplier system $\chi$, whose zeroes and poles are supported on the boundary and which at the cusp $(\infty, \infty)$ has the convergent product expansion 
	$$\Psi(\tau_1, \tau_2) = C q_1^{\rho_{\tilde{\kappa}'}}
	\prod_{\lambda = 1}^\infty \prod_{\beta \in \Z/N\Z} \left(1 - \mathbf \zeta_N^\beta q_1^{\lambda} \right)^{c_{(0,\beta,\lambda,0)}} \cdot  q_{2}^{\rho_{\tilde{\kappa}}}
	\prod_{\lambda = 1}^\infty \prod_{\beta \in \Z/N\Z} \left(1 - \mathbf \zeta_N^\beta q_{2}^{\lambda/N'} \right)^{c_{(0,\beta,0,\lambda)}}$$
	for some non-zero constant $C \in \C^\times$.
\end{Satz}

\begin{proof}
	This follows from \cite[Theorem 13.3]{Bor98} by plugging in the constructions above. First note that the cusp $(\infty, \infty)$ corresponds to the primitive isotropic vector $\kappa$. As all the Fourier coefficients of $F$ with negative indices vanish, we get that the divisor of $\Psi$ on $\HH \times \HH$ vanishes and the zeroes and poles are supported on the boundary.
	Moreover, for the product expansion, note that we only have to consider $\lambda \in K'$ with $Q(\lambda) = 0$, which means that for $\lambda = (\lambda_1, \lambda_2) \in K'$ we have $\lambda_1 = 0$ or $\lambda_2 = 0$. This guarantees that we can split the product expansion in the form $\Psi(\tau_1, \tau_2)  = \psi_1(\tau_1) \psi_2(\tau_2)$ for functions $\psi_i \colon \HH \to \C$.
\end{proof} 

\section{Invariants of the Weil Representation}

Let $D$ be a discriminant form of level $N$.
The main goal of this section is to prove that the space of invariants $\C[D]^{\MPZ}$ under the Weil representation is generated by the characteristic functions of self-dual isotropic subgroups of $D$ (as defined below) provided that $D$ admits a self-dual isotropic subgroup.
Let $H \subset D$ be a subset. We define its characteristic function $v^H \in \C[D]$ to be
\index{characteristic function}
$$ v^H = \sum_{\gamma \in H} \mathfrak{e}_\gamma.$$
We set $H^\bot = \{\gamma \in D \colon B(\gamma, \gamma') = 0 \text{ for all } \gamma' \in H \}$.
We call a subgroup $H$ of $D$ \emph{self-orthogonal} if $H \subseteq H^\bot$, \emph{self-dual} if $H = H^\bot$, and \emph{isotropic} if the restriction of the quadratic form $Q|_{H}$ vanishes on $H$.
We say that $H$ is \emph{co-isotropic} if $H^\bot$ an isotropic subgroup of $D$.

\cite[Proposition 5.8]{Scheit} shows that the characteristic function of a subgroup $H$ of $D$ is invariant under the Weil representation if and only if $H$ is self-dual and isotropic. Moreover, if $D$ contains a self-dual isotropic subgroup, then $\sign(D) = 0$ and $|D|$ is a square. In particular, the Weil representation factors through $\SLZ$ in this case. 
\begin{Satz}
	\label{IsoSubIn}
	Assume that  $D$ possesses self-dual isotropic subgroups. Then $\C[D]^{\SLZ}$ is spanned by the characteristic functions $v^H$ of self-dual isotropic subgroups $H$ of $D$.
\end{Satz}

In various places (e.g., \cite{EhlSkor} and \cite{Skor08}) this fact is mentioned but not shown. However, a proof in the upcoming reference \cite{Skor16} is advertised. As suggested in \cite{Skor08}, we adapt the proof of a similar result on codes from \cite[Theorem 5.5.7]{NRS}. We stipulate for the remainder of this section that $|D|$ is a square and $\sign(D) = 0$.

We need the following observation. Whenever we have an orthogonal decomposition of our discriminant form, we can describe which subgroups can occur as the projections of self-dual isotropic subgroups to the components:

\begin{Lemma}
	\label{selfOrthog}
	Let $D$ be a discriminant form with an orthogonal decomposition $D = D_1 \oplus D_2$, with the corresponding projections $\pi_1$ and $\pi_2$. 
	Assume that $H$ is a self-dual isotropic subgroup of $D$. Then $H_1 = \pi_1(H) \subset D_1$ and $H_2 = \pi_2(H) \subset D_2$ are co-isotropic.
	Moreover, we have $|H_1||H_2^\bot| = |H_2| |H_1^\bot| = |H|$.
	In particular, $H_1$ is self-dual if and only if $H_2$ is self-dual, in which case $H = H_1 \oplus H_2$.
\end{Lemma}
\begin{proof}
	Let $\gamma \in H_1^\bot$. Then $(\gamma, 0)$ satisfies 
	$$B((\gamma,0),(\beta_1, \beta_2)) = B_{D_1}(\gamma, \beta_1) + B_{D_2}(0, \beta_2) = 0$$
	for all $(\beta_1, \beta_2) \in H$. Thus, we get $(\gamma,0) \in H^\bot = H$, $\gamma \in H_1$ and $Q_{D_1}(\gamma) = Q((\gamma,0)) = 0$. This means that $Q$ vanishes on $H_1^\bot$, in other words, $H_1^\bot$ is isotropic. Similarly for the second component.
	
	It follows that the set of those $\gamma \in D_1$ such that $(\gamma,0) \in H$, in other words the kernel of the projection $\pi_2 \colon H \to H_2$, is exactly $H_1^\bot$. By comparing the cardinality of $H$ with the kernel and image of $\pi_2$, we find $|H| = |H_1^\bot||H_2|$. The other equality can be shown similarly. 
	
	For the last statement assume that $H_1$ is self-dual. In particular, $|H_1| = |H_1^\bot|$. Hence, we get $|H_2| = |H_2^\bot|$. As $H_2^\bot \subseteq H_2$ by the first part, this shows that $H_2$ is self-dual. The other direction follows similarly. By the calculation above, we see that $H_1 \oplus H_2 = H_1^\bot \oplus H_2^\bot \subseteq H$ in this case. But $|H_1 \oplus H_2| = |H_1| \cdot |H_2| = |H|$. This proves the claim. 
\end{proof}

We can use this lemma to show that we can also restrict ourselves to the case where the level $N$ of $D$ is a prime power. Let $N = \prod_{i = 1}^n p_i^{\nu_i}$. By \cite[Proposition 3.1]{EhlSkor}, the decomposition $D = \bigoplus_{p_i} D_{p_i}$ into its $p_i$-subgroups is orthogonal with respect to $Q$. This decomposition induces an isomorphism on the space of invariants
$\C[D]^{\SL(\Z/N\Z)} \simeq \bigotimes_{i = 1}^n \C[D_{p_i}]^{\SL(\Z/p_i^{\nu_i}\Z)}.$
It thus suffices to prove Theorem \ref{IsoSubIn} when $N$ is a prime power due to the next lemma.
\begin{Lemma}
	\label{pSubGrpDecom}
	The self-dual isotropic sugroups of $D$ are exactly the subgroups of the form 
	$$ H = \bigoplus_{i = 1}^n H_{p_i} $$
	for self-dual isotropic subgroups $H_{p_i}$ of $D_{p_i}$.
\end{Lemma}
\begin{proof}
	For the decomposition of self-dual isotropic subgroups we first consider an orthogonal decomposition $D = D_1 \oplus D_2$ where the orders of $D_1$ and $D_2$ are coprime. Let $H$ be a self-dual isotropic subgroup and $H_1, H_2$ be the projections to $D_1$ and $D_2$ respectively. From Lemma \ref{selfOrthog} we know that $|H_1||H_2^\bot| = |H_2| |H_1^\bot|$. Thus, both $H_1$ and $H_2$ are self-dual since their orders are coprime, and hence $H = H_1 \oplus H_2$ by the last statement of Lemma \ref{selfOrthog}.
	Using this argument inductively completes the proof.
\end{proof}

As the next step, we show that characteristic functions of isotropic subgroups generate the space of invariants under the Weil representation of a subgroup of $\SLZ$.
We will need matrices \(M_u\) for \(u \in (\Z/N\Z)^\times\) that operate as a certain permutation on the standard basis of \(\C[D]\) under the Weil representation. 
\begin{Lemma}
	\label{lemMu}
	For each $u \in (\Z/N\Z)^\times$, there exists a matrix $M_u \in \SLZ$ that operates as 
	$$\varrho_D(M_u) \ef_\gamma = \ef_{u \gamma}$$
	for every isotropic $\gamma \in D$.
\end{Lemma}
\begin{proof}
	Let $M_u$ be a matrix of the form $\sMa * * N u \in \SLZ$. Such a matrix exists as $N$ and $u$ are coprime.
	The assertion is now an immediate consequence of  \cite[Proposition 14.5.11]{CoStr} and  \cite[Proposition 4.5]{Scheit} that analyse the action of elements of $\Gamma_0(N)$ under the Weil representation. The conditions that $\sign(D) = 0$ and that $|D|$ is a square guarantee that the additional factors reduce to $1$. 
\end{proof}

We assume for the remainder of this section that the level $N$ is a prime power, i.e., $N = p^\nu$ for a prime number $p$.
Using the matrices $M_u$ we can characterise the space of invariants under the Weil representation of a certain subgroup of \(\SLZ\).
\begin{Proposition}[cf. {\cite[Theorem 5.1.3 and Remark 5.1.5]{NRS}}]
	\label{genPar}
	The vectors $v^H$, where $H$ ranges over isotropic subgroups of $D$, generate the subspace of invariants under the Weil representation of the subgroup of $\SLZ$ generated by $T$ and $M_u$ for $u$ ranging over $(\Z/N\Z)^\times$.
\end{Proposition}
\begin{proof}
	Recall that by definition $T$ operates as multiplication by $\eb(Q(\gamma))$ on $\ef_\gamma$ and thus $\ef_\gamma$ is invariant under $T$ if and only if $\gamma$ is isotropic.
	Moreover, characteristic functions of isotropic subgroups are invariant under $M_u$ since for any subgroup $H$ of $D$ multiplication by an integer coprime to the order of $H$ is an isomorphism on $H$.
	
	For the other direction note that by Lemma \ref{lemMu} the set of images of $\ef_\gamma$ under $M_u$, for $u$ ranging over $(\Z/N\Z)^\times$, is the set of $\ef_{u\gamma}$ for $u \in(\Z/N\Z)^\times$,  and those are exactly those $\ef_{\beta}$ with $\beta \in D$ such that the cyclic subgroups of $D$ generated by $\beta$ and $\gamma$ coincide.
	Using again that $\ef_\gamma$ is invariant under $T$ if and only if $\gamma$ is isotropic, this shows that the characteristic functions $v^{H^*}$ of the set of generators
	$H^* = \{ \gamma \in H \colon \langle \gamma \rangle = H \}$ of cyclic, isotropic subgroups $H$ of $D$ generate the space of invariants. 
	
	Let now $H$ be a cyclic, isotropic subgroup with generators $H^* \subset H$ and take $\gamma \in H^*$. Let $(p) \subset \Z/p^\nu\Z$ be the ideal generated by $p$. We have $H^* = H \backslash (p) H$, and $H' = (p)H  = H\backslash H^\ast$ is again a cyclic, isotropic subgroup. Thus, 
	$$v^{H^*} = v^H - v^{H'},$$
	and the characteristic functions of isotropic subgroups generate the subspace of invariants under the Weil representation of the subgroup of $\SLZ$ generated by $T$ and $M_u$ for $u$ ranging over $(\Z/N\Z)^\times$.
\end{proof}

We now study the invariants of the endomorphism $A = \left(\frac{1}{N} \sum_{n = 0}^{N-1} \varrho_D(T)^n \right) \circ \varrho_D(S)$. Note that any invariant of $\SLZ$ under the Weil representation is an invariant of $A$.
We explicitly compute the operation of $A$ on characteristic functions of isotropic subgroups.
\begin{Lemma}[cf. {\cite[Lemma 5.5.9]{NRS}}]
	\label{XS}
	For any isotropic subgroup $H$ of $D$ there are rational numbers $n_{H'}$ such that
	$$ A v^H = \sum_{H'}
	n_{H'} v^{H'}, $$
	where the sum is taken over isotropic subgroups $H' \subseteq H^\bot$ that are generated by $H$ and a single element of $H^\bot$. 
	Moreover, $n_H = 1$ if and only if $H$ is self-dual.
\end{Lemma}
The coefficients $n_{H'}$ in the lemma are unique as can be seen as follows. Let $H'$ be a maximal such subgroup and let $\gamma \in H^\bot$ be such that $\langle H, \gamma \rangle = H'$. By the maximality of $H'$, the element $\gamma$ is not contained in any other subgroup of this form. But this means that $n_{H'}$ is the coefficient of $\gamma$ in $Av^H$. One concludes inductively that also all other coefficients are uniquely determined. 
\begin{proof}
	Let $H \subset D$ be an isotropic subgroup.
	By the proof of \cite[Proposition 5.8]{Scheit}, the operation of $\varrho_D(S)$ on the characteristic function of $H$ is given by 
	$$ \varrho_D(S) v^H = \frac{|H|}{\sqrt{|D|}} v^{H^\bot},$$
	where we use again the fact that $\sign(D) = 0$. 
	Thus, the operation of $A$ is given by
	\begin{equation}
		\label{eqnMvH}
		A v^H = \frac{|H|}{\sqrt{|D|}} \sum_{\substack{\gamma \in H^\bot \\ Q(\gamma) = 0}} \ef_\gamma = \frac{|H|}{\sqrt{|D|}} \left(  v^H + \sum_{\substack{C \subset H^\bot, C \not \subset H \\ C \text{ cyclic, isotropic}}} v^{C^*} \right),
	\end{equation}
	where we denote by $C^* = \{ \gamma \in C \colon \langle \gamma \rangle = C \}$ the set of generators of $C$ as above.
	This proves the claim in the case where $H$ is self-dual.
	
	Assume now that $H$ is not self-dual. Let now $H' =  \langle H, \gamma \rangle$ be the subgroup of $H^\bot$ generated by $H$ and some isotropic $\gamma \in H^\bot \setminus H$. 
	We claim that $H'$ contains a unique maximal subgroup containing $H$, namely $H'' = \langle H, p\gamma \rangle$. 
	In order to prove this claim, let $ \beta \in H'$. We can write $\beta = \delta +  n\gamma$ for some $\delta \in H$ and $n \in \Z/N\Z$. If $n \in (p)$, we clearly have $\beta \in H''$. Otherwise, $n \in (\Z/N\Z)^\times$ as $(p)$ is the unique maximal ideal in $\Z/N\Z$. Then, $\gamma = n^{-1}(\beta - \delta) \in \langle H, \beta \rangle$ and hence $H$ and $\beta$ generate $H'$. 
	This proves the claim.

	Thus, any element $\beta \in H' \setminus H''$ generates an isotropic cyclic subgroup $C$ of $H^\bot$ such that the union of $H$ and $C$ generates $H'$.   
	Hence, 
	$$ \sum_{\substack{C \subset H^\bot, \langle H, C\rangle = H' \\ C \text{ cyclic, isotropic} }} v^{C^*} = v^{H'} - v^{H''}.$$
	Together with (\ref{eqnMvH}) this shows that we can write $ A v^H$ as claimed. 
	
	To see that $n_H \neq 1$ whenever $H$ is not self-dual, observe that the coefficient is given by 
	$ n_{H}  = \frac{|H|}{\sqrt{|D|}}\cdot (1 - \tilde{n}_H)$,
	where $\tilde{n}_H$ is the number of subgroups $ H' \subset H^\bot$ such that  $ H' = \langle H, \gamma \rangle$ for some isotropic $\gamma \in H^\bot$ and such that $H$ is a maximal subgroup of $H'$. In particular, $n_H < 1$.
\end{proof}

Thus, Lemma \ref{XS} means that $A$ operates on the space of invariants of the subgroup of Proposition \ref{genPar}, which is spanned by the set of characteristic functions $v^H$ for isotropic subgroups $H$ of $D$, and acts triangularly on this spanning set. 
We can therefore apply the following lemma, that is proven in \cite[Lemma 5.5.10]{NRS}.

\begin{Lemma}
	\label{poSetBasis}
	Let $V$ be a finite dimensional vector space (over an arbitrary field), let $A$ be a linear transformation on $V$, and let $(P, \leq)$ a partially ordered set. Suppose that there exists a spanning set $v_p$ of V indexed by $p \in P$ on which $A$ acts triangularly, i.e.
	$$A v_p = \sum_{q \geq p} c_{pq} v_q, $$
	for suitable coefficients $c_{pq}$. Suppose furthermore that $c_{pp} = 1$ if and only if $p$ is maximal in $P$. Then the subspace of $V$ that is fixed by $A$ is spanned by the elements $v_p$ for $p$ maximal.
\end{Lemma}

The set of isotropic subgroups of a given discriminant form is partially ordered by inclusion. We show that the maximal elements in this partial order are exactly the self-dual isotropic subgroups.
\begin{Lemma}
	\label{lemSelfDualMax}
	If $D$ possesses self-dual isotropic subgroups, then every isotropic subgroup $H \subset D$ is contained in some self-dual isotropic subgroup.
	Moreover, the maximal isotropic subgroups of $D$ are exactly the self-dual ones in this case.
\end{Lemma}
\begin{proof}
	Let $H$ be an isotropic subgroup such that $H \neq H^\bot$, and let $\Tilde{H}$ be a self-dual isotropic subgroup. Consider the group homomorphism $\Tilde{H} \to D/H^\bot$. Since $|D/H^\bot| < |\tilde{H}|$, the kernel contains a non-trivial element $x \in \tilde{H} \cap H^\bot$. Then the group $\langle H, x \rangle$ is still isotropic.
	This means that maximal elements in the set of isotropic subgroups of $D$ have to be self-dual.
	On the other hand, a self-dual isotropic subgroup cannot be a proper subgroup of another isotropic subgroup. This can be easily seen by comparing the orders of the groups.
\end{proof}

\begin{proof}[Proof of Theorem \ref{IsoSubIn}]
	Let $D$ be a discriminant form posessing self-dual isotropic subgroups. 
	Recall that by \cite[Proposition 5.8]{Scheit} the characteristic functions of self-dual isotropic subgroups are invariant and that our standing assumptions $\sign(D) = 0$ and $|D|$ a square are satisfied. Moreover, by \cite[Proposition 3.1]{EhlSkor} and Lemma \ref{pSubGrpDecom} we may assume that the level $N$ of $D$ is a prime power.
	
	We consider the space $V$ of invariants under $T$ and $M_u$ for $u \in (\Z/N\Z)^\times$. 
	Any invariant for the full group $\SLZ$ clearly belongs to this space and is moreover invariant under the operation of $A$ defined before Lemma \ref{XS}.
	By Proposition \ref{genPar}, $V$ is spanned by the characteristic functions of isotropic subgroups. 	As discussed above, the set isotropic subgroups of $D$ is partially ordered with respect to the inclusion relation.
	By Lemma \ref{XS}, the map $A$ restricts to an endomorphism of $V$ and operates triangularly on the characteristic functions of isotropic subgroups.
 	By Lemma \ref{lemSelfDualMax}, the maximal elements in this partially ordered set are precisely the self-dual isotropic subgroups.
	By Lemma \ref{poSetBasis}, the space $V^A$ of vectors that are invariant under $A$ is thus spanned by the characteristic functions of self-dual isotropic subgroups. As $\C[D]^{\SLZ}$ is contained in $V^A$ and contains all characteristic functions of self-dual isotropic subgroups, $\C[D]^{\SLZ}$ is spanned by the characteristic functions of self-dual isotropic subgroups.
\end{proof} 

\section{Example: The Lattice $L_{N,N'}$}
\label{secLN}
We want to describe the space of invariants of the Weil representation for the discriminant form of the orthogonal sum $L_{N,N'} = U(N) \oplus U(N')$ of two rescaled hyperbolic planes for a divisor $N'$ of $N$.
Recall that the associated discriminant group is given by 
\[ D := D_{N,N'} = (\Z/N\Z)^2 \oplus (\Z/N'\Z)^2,\]
\[ Q \colon D_{N,N'} \to \Q/\Z, \qquad (w,x,y,z) \mapsto \frac{wx}{N} + \frac{yz}{N'} \mod \Z.\]
We also use the shorthand notation $D_N = D_{N,1} = (\Z/N\Z)^2$, so $D \cong D_N \oplus D_{N'}$. 
In order to distinguish the two summands of $D$ in the case $N = N'$, we denote the first summand of $D$ by $D_N$ and the second summand by $D_{N'}$.

By Lemma \ref{selfOrthog} applied to the orthogonal decomposition $D_{N,N'} = (\Z/N\Z)^2 \oplus (\Z/N'\Z)^2$, the projection of a self-dual isotropic subgroup $H$ of $D_{N,N'}$ to $D_N$ is a co-isotropic subgroup of $D_N$.
We thus start by characterising co-isotropic subgroups of $D_N$. 
This will already yield a complete description of the invariants in the case $N' = 1$. 
We use these observations to construct a reasonably large class of self-dual isotropic subgroups for general $N'$ and show that our list is exhaustive when $N'$ is a prime number.

In order to characterise co-isotropic subgroups of $D_N$, we use the following notation for subgroups of $D_N$. Let $x$ and $z$ be two divisors of $N$ and let $y \in \Z/N\Z$.
We define
$H_{x,y,z}$ to be the subgroup of $D_N$ generated by $(x,y)$ and $(0,z)$. In particular, for fixed $x$ and $z$ the group $H_{x,y,z}$ only depends on the class of $y$ in $\Z/z\Z$. 

\begin{Lemma}
	\label{lemSubgrD}
		Every subgroup of $D_N$ is of the form $H_{x,y,z}$ for suitable divisors $x$ and $z$ of $N$, and a suitable $y \in \Z/NZ$. Moreover, $z$ can be chosen to be the minimal (positive) divisor of $N$ such that $(0,z) \in H_{x,y,z}$. 
		For this minimal choice of $z$ it follows that 
		\begin{equation}
			|H_{x,y,z}| = \frac{N}{x} \cdot \frac{N}{z},
			\label{eqnHxyz}
		\end{equation}
		that $xz$ divides $Ny$, and that the orthogonal complement of $H_{x,y,z}$ in $D_N$ is given by 
		$$H_{x,y,z}^\bot = H_{\frac N z , -\frac{Ny}{xz}, \frac{N}{x}}. $$
\end{Lemma}
\begin{proof}
	Let $G \subseteq D_N$ be a subgroup. Then $G = H_{x,y,z}$ with 
	\begin{align*}
		x & =  \min \{ 1 \leq s \leq N \colon (s, r) \in G \text{ for some } r \in \Z/N\Z \}, \\ 
		y & =  \min \{ 0 \leq r \leq N-1 \colon (x, r) \in G \}, \text{ and } \\ 
		z & =  \min \{ 1 \leq r \leq N \colon (0, r) \in G \}.
	\end{align*}
	It follows directly that $x$ and $z$ are divisors of $N$ and that $y < z$. Moreover, $\frac{N}{x} (x,y) = (0,\frac{Ny}{x}) \in H_{x,y,z}$, which by the minimality of $z$ implies that $xz$ divides $Ny$. This also proves Equation (\ref{eqnHxyz}).
	
	For the description of the orthogonal complement we write $H^\bot_{x,y,z} = H_{x',y',z'}$. We obtain $0 = B((x,y),(0,z')) = B((0,z),(x',y'))$ , i.e. $xz' \equiv zx'  \equiv 0 \mod N$. Using Equation (\ref{eqnHxyz}), we obtain
	$$ \frac{N^2}{x'z'} = |H_{x',y',z'}| = \frac{N^2}{|H_{x,y,z}|} = xz.$$
	This means that $x' = N/z$ and $z'=N/x$.
	From 
	$0 = B((x,y),(x',y')) = \frac{xy' + x'y}{N} \mod \Z$ we obtain that $x | x'y = \frac{Ny}{z}$. In particular, $y' \equiv - \frac{Ny}{xz} \mod \frac N x$.
\end{proof}

Note that the choice of the divisor $z$ is not unique in general. Namely, when $\frac{Ny}{x}$ equals (up to a unit in $\Z/N\Z$) the minimal possible choice $z_0$ for $z$ (as constructed in the proof of the previous lemma), choosing $z$ to be any multiple of $z_0$ clearly gives rise to the same subgroup $H_{x,y,z} = H_{x,y,z_0}$ of $D_N$.
By the previous lemma, we may and will always assume without further mention that $z$ is chosen minimally.
Note that the proof of the previous lemma also shows that for the orthogonal complement $H_{x',y',z'} = H_{x,y,z}^\bot$ the minimal choice is $z' = \frac{N}{x}$. 
Using this description of subgroups of $D_N$, we can classify all (self-dual or co-) isotropic subgroups of $D_N$.

\begin{Lemma}
	\label{lemSubgrD1}
	\begin{enumerate}
		\item $H_{x,y,z}$ is an isotropic subgroup if and only if $N|xz$ and $N|xy$.
		\label{lem3}
		\item $H_{x,y,z}$ is co-isotropic if and only if $xz | N$ and $xz^2 | Ny$.
		\label{lem4}
		\item The group $H_{x,y,z}$ is self-dual isotropic if and only if $x z = N$ and $y \equiv 0 \mod z$. In other words, the self-dual isotropic subgroups of $D_N$ are exactly those of the form $H_{d,0,N/d}$ for divisors $d$ of $N$.
		\label{lem5}
	\end{enumerate}
\end{Lemma}
\begin{proof}
	For (\ref{lem3}) note that the group generated by $(x,y)$ and $(0,z)$ is isotropic if and only if both generators are isotropic and orthogonal. But this means exactly $xy, xz \equiv 0 \mod N$.
	Then (\ref{lem4}) follows directly from Lemma \ref{lemSubgrD} and (\ref{lem3}). 
	Finally, (\ref{lem5}) is a consequence of (\ref{lem4}) together with the fact that for self-dual isotropic subgroups $|H_{x,y,z}| = N^2/(xz) = N$ by Equation (\ref{eqnHxyz}).
\end{proof}

In the case $N' =1$ we show that the characteristic functions of self-dual isotropic subgroups are also linearly independet. This gives a new proof of \cite[Lemma 3.2]{Ye2021}, see also \cite[Corollary 5.5]{Zemel2021}.
\begin{Korollar}
	\label{corBasis}
	The characteristic functions $v^{H_{d,0,N/d}}$ where $d$ is a divisor of $N$ form a basis of the space 
	$\C[D_N]^{\SLZ}$. In particular, $\dim \C[D_N]^{\SLZ} = \sigma_0(N)$, where $\sigma_0(N)$ denotes the number of divisors of $N$. 
\end{Korollar}
\begin{proof}
	The characteristic functions span the space of invariants following Theorem \ref{IsoSubIn}. They are linearly independent since for every pair of divisors $d, d'$ of $N$ we have $(d, N/d) \in H_{d',0,N/d'}$ if and only if $d' = d$.
	The second assertion follows immediately.
\end{proof}

As another immediate consequence of Lemma \ref{lemSubgrD1}, we have the following easier description of co-isotropic subgroups of $D_N$ when $N$ is square-free. This will be used in the proof of Corollary \ref{corSelfDualNp} below.
\begin{Lemma}
	\label{lemCoIsoP}
	Let $N$ be square-free. The co-isotropic subgroups of $D_N$ are the subgroups $H_{d_1,0,d_2}$ for a pair of divisors $d_1, d_2$ of $N$ such that $d_1 d_2|N$.
\end{Lemma}
\begin{proof}
	Let $d_1, d_2$ be a pair of divisors of $N$ such that $d_1 d_2|N$. Then $H_{d_1,0,d_2}$ is co-isotropic by Lemma \ref{lemSubgrD1} (\ref{lem4}).
	
	Let $H_{x,y,z}$ be a co-isotropic subgroup of $D_N$. 
	By Lemma \ref{lemSubgrD1} (\ref{lem4}), we obtain $z | \frac{N}{xz} y$. Since $N$ is square-free, $z$ and $\frac{N}{xz}$ are coprime, so $z|y$. Thus, we may choose $y =0$.
\end{proof}

The next goal is to give conditions on how to obtain self-dual isotropic subgroups $H$ of $D = D_{N,N'} = D_N \oplus D_{N'}$ from a co-isotropic subgroup of $D_N$ and one of $D_{N'}$.

Let \(H\) be a self-dual isotropic subgroup of \(D\). Then $|H|^2 = |D| = N^2 N'^2$. 
We denote the projections of $H$ onto the two components of $D$ by $H_1$ and $H_2$. By Lemma \ref{selfOrthog}, we get that $H_1$ and $H_2$ are co-isotropic subgroups of $D_N$ and $D_{N'}$ respectively.
By Lemma \ref{lemSubgrD} and  Lemma \ref{lemSubgrD1} (\ref{lem4}), we can write $H_1 = H_{x,y,z}$ and $H_2 = H_{x',y',z'}$ for parameters such that $xz|N$ and $x'z'|N'$.

Again by Lemma  \ref{selfOrthog}, we find that $|H_1| |H_2^\bot| = |H_2||H_1^\bot| = N N'$. 
Together with $|H_1||H_1^\bot| = N^2$ and $  |H_2||H_2^\bot| = (N')^2$, we get that $|H_1| = \frac{N}{N'} |H_2|$. Now Equation (\ref{eqnHxyz}) implies that $N' xz = N x'z'$.

In order to describe self-dual isotropic subgroups of $D$, we introduce the following notation.
For any pair of co-isotropic subgroups $H_{x,y,z}, H_{x',y',z'}$ in $D_N$ and $D_{N'}$ respectively such that $N'xz = N x'z'$, and any pair of elements $(a,b), (c,d) \in H_{x',y',z'}$, we consider the subgroup
$$H_{(x,y,z),(x',y',z')}^{(a,b),(c,d)} = \left \langle (x,y,a,b), (0,z,c,d), \left (0,0, \frac{N'}{z'}, -\frac {N'y'}{x'z'} \right), \left(0, 0, 0, \frac{N'}{x'} \right) \right \rangle$$
of $H_{x,y,z} \oplus H_{x',y',z'} \subset D$.
Note that the latter two elements are generators of $\{ (0,0) \} \oplus H^\bot_{x',y',z'}$ by Lemma \ref{lemSubgrD}. 
Recall that by the proof of Lemma \ref{selfOrthog}, this subgroup has to be contained in $ H_{(x,y,z),(x',y',z')}^{(a,b),(c,d)}$ in order for $H_{(x,y,z),(x',y',z')}^{(a,b),(c,d)}$ to be self-dual isotropic. 
By definition $H_{(x,y,z),(x',y',z')}^{(a,b),(c,d)}$ does not change when replacing $(a,b)$ (respectively $(c,d)$) by an element of its class in $H_{x',y',z'}/H_{x',y',z'}^\bot$.

In general, the second projection of $H_{(x,y,z),(x',y',z')}^{(a,b),(c,d)}$ might be a proper subgroup of $H_{x',y',z'}$. 
However, every self-dual isotropic subgroup $H$ of $D$ can be written as $H = H_{(x,y,z),(x',y',z')}^{(a,b),(c,d)}$ such that $\pi_2(H) = H_{x', y',z'}$ due to the following lemma.

\begin{Lemma}
	\label{lemReprHselfdual}
	For every self-dual isotropic subgroup $H$ of $D = D_{N,N'}$ such that $\pi_1(H) =  H_{x,y,z}$ and $\pi_2(H) = H_{x',y',z'}$, there is a a pair of elements $(a,b), (c,d)$ of $H_{x',y',z'}$ such that $H = H_{(x,y,z),(x',y',z')}^{(a,b),(c,d)}$.
	Both $(a,b)$ and $(c,d)$ are unique in $H_{x',y',z'}/H^\bot_{x',y',z'}$.
\end{Lemma}
\begin{proof} 
	Let $(a,b), (c,d) \in \pi_2(H)$ such that $(x,y,a,b), (0,z,c,d) \in H$. 
	It is clear that
	$$H_{(x,y,z),(x',y',z')}^{(a,b),(c,d)} \subseteq H.$$ For the other inclusion let $\gamma = (\gamma_1, \gamma_2, \gamma_3, \gamma_4) \in H$. Since $(x,y), (0,z)$ generate $\pi_1(H)$ we find $(\gamma_3', \gamma_4') \in \pi_2(H)$ such that $(\gamma_1, \gamma_2, \gamma_3', \gamma_4') \in H_{(x,y,z),(x',y',z')}^{(a,b),(c,d)}$. But then $(\gamma_3, \gamma_4) - (\gamma_3', \gamma_4') \in \pi_2(H)^\bot$ by Lemma \ref{selfOrthog} (and its proof). Thus, $\gamma \in H_{(x,y,z),(x',y',z')}^{(a,b),(c,d)}$.
	
	For another pair of elements $(a',b') \in \pi_2(H)$ with $(x,y,a',b') \in H$ we obtain $(a-a', b-b') \in H_{x',y',z'}^\bot$ by the same argument. Similarly for $(c,d)$.
\end{proof} 

The following lemma gives a criterion for such a subgroup to be self-dual isotropic.
\begin{Lemma}
	\label{condSum}
	For co-isotropic subgroups $H_{x,y,z} \subset D_N$ and $H_{x',y',z'} \subset D_{N'}$ with $N'xz = Nx'z'$, the group 
	$ H_{(x,y,z),(x',y',z')}^{(a,b),(c,d)}$
	is a self-dual isotropic subgroup of $D$ if and only if 
	\begin{eqnarray}
		\label{condABiso}
		\frac{N}{N'}ab + xy & \equiv & 0 \mod N, \\
		\label{condCDiso}
		cd & \equiv & 0 \mod N', \\
		\label{condOrth}
		\frac{N}{N'} (ad + bc) + xz & \equiv & 0 \mod N.
	\end{eqnarray}
\end{Lemma}
\begin{proof}
	We abbreviate $H = H_{(x,y,z),(x',y',z')}^{(a,b),(c,d)}$.
	The conditions exactly mean that both $(x,y,a,b)$ and $(0,z,c,d)$ are isotropic and orthogonal. Thus, the "only if" part is clear. For the "if" part, we note that the conditions (\ref{condABiso})-(\ref{condOrth}) imply that $H$ is isotropic.  
	Moreover, for every $\gamma \in H_{x,y,z}$ the set of elements $\gamma' \in H_{x',y',z'}$ such that \( (\gamma, \gamma') \in H \) is a coset of $\pi_2(H)^\bot$.
	As $H_{x',y',z'}^\bot \subseteq \pi_2(H)^\bot$ by construction, there are at least $|H_{x',y',z'}^\bot|$ many elements $\gamma' \in H_{x',y',z'}$ such that \( (\gamma, \gamma') \in H \). Hence, $|H| \geq |H_{x,y,z}||H^\bot_{x',y',z'}| = \frac{N^2}{xz} \cdot \frac{x'z'}{(N')^2} = N N'$, where we used the assumption $N'xz  = Nx'z'$ in the last step. Since $H$ is isotropic, a cardinality argument shows that $|H^\bot| = |H| = NN'$, so $H = H^\bot$. In other words, $H$ is self-dual. 
\end{proof}

We now give examples for self-dual isotropic subgroups of $D_{N,N'}$.

\begin{Beispiel}
	\label{exNy0}
	For $u \in (\Z/N'\Z)^\times$  and two pairs of divisors $x, z$ of $N$ and $x', z'$ of $N'$ such that $xz | N$, $x'z' | N'$, and $N' xz = N x'z'$, the groups 
	$$H_{(x, 0, z),(x', 0, z')}^{(ux', 0), (0, -u^{-1}z')} \qquad \text{ and } \qquad H_{(x, 0, z),(x', 0, z')}^{(0, uz'), (-u^{-1}x', 0)},$$
	are self-dual isotropic subgroups of $D_{N,N'}$.
	The conditions from Lemma \ref{condSum} are easily verified.
\end{Beispiel}

Constructing examples for subgroups $H_{x,y,z}$ with $y \neq 0$ seems to be more difficult. We construct the following ones in the case $N = N'$.
\begin{Beispiel}
	\label{exy-y}
	Let $H_{x,y,z}$ be a co-isotropic subgroup of $D_N$. Then $H_{x,-y,z}$ is co-isotropic according to Lemma \ref{lemSubgrD1}, and for a unit $u \in (\Z/N\Z)^\times$ we have that
	$ H_{(x,y,z),(x,-y,z)}^{ (-ux,u^{-1}y), (0,u^{-1}z)} $
	is a self-dual isotropic subgroup of $D_{N,N}$ by Lemma \ref{condSum}.
\end{Beispiel}
Note that for $y = 0$ this recovers the first family of self-dual isotropic subgroups of $D_{N,N}$ constructed in Example \ref{exNy0}, while for $y \neq 0$ the subgroup constructed in Example \ref{exy-y} is distinct from the ones of  Example \ref{exNy0}.
The natural question that arises is whether or not this list of subgroups is already complete. 
We show in Proposition \ref{exy0} below that when $N'$ is a prime power Example \ref{exNy0} gives all self-dual isotropic subgroups coming from co-isotropic subgroups of the form $H_{x,0,z}$ (that means with $y = 0$) and that when $N'$ is prime other cases cannot occur.
In general, there will be self-dual isotropic subgroups that are not covered by our examples above.

\begin{Lemma}
	Let $H$ be a self-dual isotropic subgroup of $D_{N,N'}$ such that one of the projections $\pi_i(H)$ for $i = 1,2$ is self-dual isotropic. 
	Then the other projection is self-dual isotropic as well and $H = \pi_1(H) \oplus \pi_2(H)$. 
	In particular, $H$ is one of the groups given in Example \ref{exNy0}. 
	\label{lemProjSelfD}
\end{Lemma}
\begin{proof}
	By Lemma \ref{selfOrthog}, if one of the projections is self-dual, the other projection is self-dual as well and $H = \pi_1(H) \oplus \pi_2(H)$. 
	By Lemma \ref{lemSubgrD1} (\ref{lem5}), there are two pairs of divisors $x,z$ of $N$ and $x', z'$ of $N'$ such that $\pi_1(H) = H_{x,0,z}$ and $\pi_2(H) = H_{x',0, z'}$. Hence, $H = H_{x,0,y} \oplus H_{x',0,y'} = H_{(x,0,z),(x',0,y')}^{(a,b),(c,d)}$ for any choice of elements $(a,b),(c,d) \in H_{x',0,y'}$ (as $H_{x',0,y'}$ is self-dual isotropic).
	In particular, $H = H_{(x,0,z),(x',0,y')}^{(x',0),(0,-z')} $ is one of the groups from Example \ref{exNy0}.
\end{proof}

\begin{Proposition}
	\label{exy0}
	Let $N$ be a positive integer and let $N'$ be some divisor of $N$, such that $N' = p^{n'}$ is a prime power. 
	Let $H$ be a self-dual isotropic subgroup of $D_{N, N'}$ with $\pi_1(H) = H_{x, 0, z}$ for a pair of divisors $x, z$ of $N$ or $\pi_2(H) = H_{x', 0, z'}$ for a pair of divisors $x', z'$ of $N'$. Then $H$ is one of the groups given in Example \ref{exNy0}.
\end{Proposition}
\begin{proof}
	By Lemma \ref{lemReprHselfdual}, the group $H$ admits a presentation as $H = H^{(a,b),(c,d)}_{(x,y,z),(x',y',z')}$. 
	The case where one of the projections $\pi_1(H)$ and $\pi_2(H)$ is self-dual isotropic is Lemma \ref{lemProjSelfD}.
	 
	When none of the projections is self-dual, we have $x'z' =p^{k'}$ for some $k' < n'$ by Lemma \ref{lemSubgrD1} (\ref{lem4}) and (\ref{lem5}). 
	Let us now consider the case $y = 0$.
	Lemma \ref{condSum} then implies
	\begin{align*}
		ab &\equiv 0 \mod N',  & 
		cd &\equiv 0 \mod N' & 
		& \text{and} &
		ad + bc &\equiv -x'z' \mod N'. & &
	\end{align*}
	From the first two conditions we obtain that
	$(N')^2|abcd$ and thus  
	$N'$
	divides at least one of $ad$ and $bc$ as $N'$ is a prime power by assumption. 
	We assume that $N'$ divides $bc$, and the other case follows analogously.
	By the third condition, $ad \equiv -x'z' \mod N'$.
	We write $a = a' p^\alpha$ for some $a' \in (\Z/N'\Z)^\times$. Then $ d \equiv -(a')^{-1} p^{k'-\alpha} \mod p^{n'-\alpha}$. In particular, $d$ is of the form $d' p^{k'-\alpha}$ for some $d' \in (\Z/N'\Z)^\times$. 
	
	We consider the subgroup $H_2$ of $D_{N'}$ generated by $(a,b)$ and $(c,d)$. By the above reformulations of conditions (\ref{condABiso}) and (\ref{condCDiso}), we get
	$ p^\alpha|c$ and similarly $p^{k'-\alpha}|b$. 
	Moreover, as $k' < n'$, the prime $p$ divides both $c/p^\alpha$ and $b/p^{k'- \alpha}$.
	Let us now consider the element of $H_2$ given by
	$$ (a,b) - \frac{(d')^{-1}b}{p^{k' - \alpha}} (c,d) = \left( a - \frac{(d')^{-1}bc}{p^{k' - \alpha}}, 0 \right) = \left( p^\alpha \left( a' - (d')^{-1} \frac{bc}{p^{k'}}\right), 0 \right).$$  
	As $a'$ is a unit in $\Z/N'\Z$ and $p$ divides $\frac{bc}{p^{k'}}$ (using the condition that $k' < n'$), we obtain that $a' - (d')^{-1} \frac{bc}{p^{k'}}$ is a unit in $\Z/N'\Z$.
	Hence, we get $(p^\alpha,0) \in H_2$ and similarly $(0, p^{k' - \alpha}) \in H_2$.
	Moreover, for any $(e,f) \in H_2$ we have $ p^\alpha | e$ and $p^{k'-\alpha} | f$. 
	But this means that $\langle (a,b), (c,d) \rangle = H_{p^\alpha, 0, p^{k'-\alpha}}$.
	
	We also know that $\langle (a,b), (c,d) \rangle \subseteq \pi_2(H)$  and thus $\pi_2(H) = H_{p^\alpha, 0, p^{k'-\alpha}}$, as both groups have the same cardinality.
	Moreover, $(0, b), (c,0) \in H_2^\bot = H_{p^{n'-k'+\alpha}, 0, p^{n' - \alpha}}$. Thus, we may choose $b = c = 0$.
	Thus, $H$ is as in Example \ref{exNy0}.
	
	Let us now assume that $\pi_2(H) =  H_{x', 0, z'}$. Let $a, b, c, d$ as in Lemma \ref{lemReprHselfdual} and \ref{condSum}, such that $H = H^{(a,b),(c,d)}_{(x,y,z),(x',0,z')}$. Then, $x'|a$ and $z'|b$ by definition. Thus, $xz|\frac{N}{N'}ab$ and by Lemma \ref{condSum} we get $xz|xy$. As $0 \leq 
	y < z$, this implies $y = 0$. The claim follows from the first assertion.
\end{proof}

\begin{Korollar}
	\label{corSelfDualNp}
	Let $N$ be a positive integer and $N' = p$ some prime factor of $N$.
	Let $H$ be a self-dual isotropic subgroup of $D_{N, N'}$. Then $H$ is one of the groups given in Example \ref{exNy0}.
\end{Korollar}

\begin{proof}
	By Lemma \ref{lemCoIsoP}, we get that $\pi_2(H)$ is of the form $H_{x', 0, y'}$. The assertion then follows from Proposition \ref{exy0}.
\end{proof}

For general positive integers $N'$, we can construct all self-dual isotropic subgroups of $D_{N,N'}$ in the $y = 0$ case from the prime factors of $N'$ as in Lemma \ref{pSubGrpDecom}. In the special case that $N'$ is square-free other cases cannot occur by Lemma \ref{lemCoIsoP}.

In contrast to the case $N' = 1$, the characteristic functions of self-dual isotropic subgroups of $D_{N,N'}$ for $N' > 1$ are in general not linearly independent.
In the case where $N'$ is prime, we can explicitly give the relations among the characteristic functions of self-dual isotropic subgroups, and thus give a basis of the space of invariants. We generalise the result from \cite{Bieker}, which only considered the case $N = N'$, as suggested in \cite[Section 5]{Zemel2021} to arbitrary $N$. For a positive integer $M$ we continue to denote by $\sigma_0(M)$ the number of its divisors.
\begin{Proposition}
	\label{propInvP}
	Let $N$ be a positive integer and $N' = p$ be a prime factor of $N$. The space $\C[D_{N,p}]^{\SLZ}$ is generated by the characteristic functions of the self-dual isotropic subgroups 
	$$ H_{d,0,N/d} \oplus H_{1, 0, p}, \hspace{1.7mm}  H_{d, 0, N/d} \oplus H_{p,0,1}, \hspace{1.7mm}  H_{(d',0,N/(pd')),(1,0,1)}^{(u,0), (0, -u^{-1})}, \hspace{1.7mm}  \text{ and } \hspace{1.7mm}  H_{(d',0,N/(pd')),(1,0,1)}^{(0, -u^{-1}), (u,0)},$$
	for divisors $d$ of $N$, divisors $d'$ of $N/p$, and $u \in (\Z/p\Z)^\times$, subject to the linear relations
	\begin{align*}
		0 & =  v^{H_{d',0,N/d'} \oplus H_{1, 0, p}}-v^{H_{d',0,N/d'} \oplus H_{p, 0, 1}} - v^{H_{pd',0,N/(pd')} \oplus H_{1, 0, p}} + v^{H_{pd',0,N/(pd')} \oplus H_{p, 0, 1}} \\
		& \hspace*{2cm} + \sum_{u = 1}^{p-1} \left( -v^{H_{(d',0,N/(pd')),(1,0,1)}^{(u,0), (0, -u^{-1})}} + v^{H_{(d',0,N/(pd')),(1,0,1)}^{(0, -u^{-1}), (u,0)}} \right) 
	\end{align*}
	for each divisor $d'$ of $N/p$.
	Moreover, any linear relation among the characteristic functions of self-dual isotropic subgroups of $D_{N,p}$ is a linear combination of the linear relations given above. In particular, $\dim(\C[D_{N,p}]^{\SLZ}) = (2p-3)\sigma_0(N/p) + 2 \sigma_0(N)$.
\end{Proposition}
\begin{proof}
	
	The space $\C[D_{N,p}]^{\SLZ}$ is spanned by the characteristic functions of the self-dual isotropic subgroups given above by Theorem \ref{IsoSubIn} and Corollary \ref{corSelfDualNp}.
	
	Let $d'$ be a divisor of $N/p$.
	In order to check the linear relation, we use the shorthand notation
	$$H^{(1)} = H_{d',0,N/d'} \oplus H_{1, 0, p}, \hspace{1.7mm} H^{(2)} = H_{d',0,N/d'} \oplus H_{p,0,1}, \hspace{1.7mm} H^{(3)} = H_{pd',0,N/(pd')} \oplus H_{1, 0, p},  $$
	$$  H^{(4)} = H_{pd',0,N/(pd')} \oplus H_{p,0,1},\hspace{1.7mm}  H^{(5)}_u = H_{(d',0,N/(pd')),(1,0,1)}^{(u,0), (0, -u^{-1})} \hspace{1.7mm} \text{and} \hspace{1.7mm}  H^{(6)}_u = H_{(d',0,N/(pd')),(1,0,1)}^{(0, -u^{-1}), (u,0)},$$
	for $u \in (\Z/p\Z)^\times$.
	We claim that
	$$ v^{H^{(1)}}-v^{H^{(2)}} - v^{H^{(3)}} + v^{H^{(4)}} + \sum_{u = 1}^{p-1} \left( -v^{H^{(5)}_u} + v^{H^{(6)}_u} \right) = 0.$$
	We check that the coefficients of $\ef_\gamma$ sum to 0 for every $\gamma \in D_{N,p}$. 
	Let $\gamma = (wd',xN/d',y,0) \in H^{(1)}$. 
	If $p | w$ and $y = 0$, then $\gamma$ is clearly contained in all of the above subgroups.
	If $p|w$ and $y \neq 0$, then $\gamma \in H^{(3)}$ and $\gamma$ is not contained in any of the other groups.
	If $p \nmid w$ and $y = 0$, we have $\gamma \in H^{(2)}$, and if $p \nmid w$ and $y \neq 0$, we find $\gamma \in H_{w^{-1}y}^{(5)}$. We argue similarly for $H^{(2)}, H^{(3)},$ and $H^{(4)}$.
	
	Let now $\gamma = (wd',xN/(pd'),y,z) \in H_u^{(5)}$ for some $u \in (\Z/p\Z)^\times$. 
	If $p|w$ or $p|x$, then necessarily $y = 0$ or $z = 0$, respectively, and we are in one of the cases discussed above. Otherwise, if $p \nmid w$ and $p \nmid x$, $\gamma$ is not contained in any of the groups $H^{(1)}, H^{(2)}, H^{(3)},$ and $H^{(4)}$. 
	Thus, $y = w u$ and $ z = -xu^{(-1)}$. We set $u' = x^{-1} w u = x^{-1}y$. Then $z = -u^{-1}x = -w(u')^{-1}$ and $y = wu = xu'$, so $\gamma \in H_{u'}^{(6)}$. 
	This establishes the linear relations among the characteristic functions.
	
	In order to prove the dimension formula, we claim that we can remove the characteristic functions of the subgroups $H^{(1)} = H_{d',0,N/d'} \oplus H_{1,0,p}$ for all divisors $d'$ of $N/p$ from the generating set to get a basis of $\C[D_{N,p}]^{\SLZ}$. The remaining set of characteristic functions still spans $\C[D_{N,p}]^{\SLZ}$ due to the linear relations that we verified above. It remains to show that this smaller set is linearly independent. But this follows from the above observations that there are enough elements that appear exactly in two of the self-dual isotropic subgroups.
	For example, by considering the coefficient of $(d', N/d', 0, 0)$, which for all $d'$ is contained precisely in the corresponding groups $H^{(1)}$ and $H^{(2)}$, we see that the coefficient of the $H^{(2)}$ in any representation of 0 as a linear combination of the remaining characteristic functions has to vanish. In a similar fashion we see that also all other coefficients vanish.  
	
	As $D_{N,p}$ admits $2 \sigma_0(N) + (2p-2) \sigma_0(N/p)$ self-dual isotropic subgroups, the number of elements in the basis of $\C[D_{N,p}]^{\SLZ}$ constructed above is given by $(2p-2)\sigma_0(N/p) + 2 \sigma_0(N)  - \sigma_0(N/p)= (2p-3)\sigma_0(N/p) + 2 \sigma_0(N)$.
\end{proof}

This immediately yields a dimension formula for $\C[D_{N,N'}]^{\SLZ}$ for square-free $N'$ using the factorisation of the space of invariants from \cite[Remark 3.2]{EhlSkor}. 
\begin{Korollar}
	For an integer $N$ with prime factorisation $N = \prod_{i = 1}^{k} p_i^{n_i}$ and a square-free divisor $N' = \prod_{i = 1}^{k'} p_i$ of $N$ for some $0 \leq k' \leq k$, the space $\C[D_{N,N'}]^{\SLZ}$ has dimension $$  \sigma_0\left( \prod_{i = k'+1}^k p_i^{n_i}\right) \cdot \prod_{i=1}^{k'} \left( (2p_i-3)\sigma_0(p_i^{n_i - 1}) + 2 \sigma_0(p_i^{n_i}) \right)=  \prod_{i=k'+1}^{k} (n_i + 1)  \cdot \prod_{i=1}^{k'} ((2p_i - 1)n_i+2).$$
\end{Korollar}
One can check that this formula recovers a special case of \cite[Theorem 5.4]{Zemel2021}, which gives a dimension formula for $\C[D]^{\SLZ}$ for a more general class of discriminant forms $D$ (including $D_{N,N'}$ for all positive integers $N$ and divisors $N'$ of $N$), compare also the discussion following Corollary 5.5 in \cite{Zemel2021}.

\section{Construction of Modular Units for $\Gamma_{L_{N,N'}}$}

We use Borcherds products to construct analogues of modular units in the setting of generalised Hilbert modular forms considered above. That means that we construct weakly holomorphic modular forms with respect to $\Gamma_{L_{N,N'}}$ (for the lattice $L_{N,N'}$ from Section \ref{secLN}) as in Subsection \ref{secLift} with their divisors supported on the boundary of the varieties considered in Subsection \ref{secModOrth} (which are essentially products of two modular curves). 
More precisely, we explicitly evaluate the Borcherds products considered in Theorem \ref{corBP2g} for the invariant vectors for the corresponding discriminant form $D_{N,N'}$ of $L_{N,N'}$ that we constructed in the previous section.

\begin{Proposition}
	\label{exLiftBasis}
	Let $H$ be one of the self-dual isotropic subgroups considered in Example \ref{exNy0}, namely
	$$H = H_{(x, 0, z),(x', 0, z')}^{(ux', 0), (0, -u^{-1}z')} \qquad \text{ or } \qquad H_{(x, 0, z),(x', 0, z')}^{(0, uz'), (-u^{-1}x', 0)},$$
	for $u \in (\Z/N'\Z)^\times$ and two pairs of divisors $x, z$ of $N$ and  $x', z'$ of $N'$ such that $x z |N$, $ x' z' | N'$, and $N'x z = N x' z'$.
	Then the Borcherds lift of its characteristic function $v^H \in \C[D_{N,N'}]^{\SLZ}$ is given by
	$$ \Psi(\tau_1, \tau_2) = C \eta \left(\frac{xz' \tau_2}{N'} - \frac{uxz}{N} \right) \cdot \eta \left( \frac{N'x}{z'} \tau_1 \right)$$
	in the first case and 
	$$ \Psi(\tau_1, \tau_2) = C \eta \left({xx' \tau_1} - \frac{uxz}{N} \right) \cdot \eta \left( \frac{x}{x'} \tau_2 \right)$$
	in the second case, where $C \in \C^\times$ is a constant.
\end{Proposition}
\begin{proof}
	We start by treating the first case $H = H_{(x, 0, z),(x', 0, z')}^{(ux', 0), (0, -u^{-1}z')}$. As before, we denote by $H_1, H_2$ the projections of $H$, i.e. $H_1 = H_{x, 0, z}$ and $H_2 = H_{x', 0, z'}$.
	
	Take $\beta_1 \in \Z/N\Z$ and $\beta_2 \in \Z/N'\Z$. 
	In order to apply Theorem \ref{corBP2g} (and Proposition \ref{WeylVect}), we analyse when  $(0, \beta_1, \beta_2, 0) \in H$ and  $(0, \beta_1, 0, \beta_2) \in H$ respectively.
	Let us assume that $(0, \beta_1, \beta_2, 0) \in H$. 
	By the definition of $H_{(x, 0, z),(x', 0, z')}^{(ux', 0), (0, -u^{-1}z')}$, we find that we can write $\beta_1 = \beta'_1 z$ for some $\tilde \beta_1$ such that $\frac{N'}{x'} | \beta'_1 z'$. Hence, $\frac{N}{xz} = \frac{N'}{x'z'} |\beta'_1$, and thus $\frac{N}{x} | \beta_1$. In a similar fashion, we show that $\frac{N'}{z'} | \beta_2$.
	Clearly, these two conditions are also sufficient for $(0, \beta_1, \beta_2, 0) \in H$.
	By a similar argument, we observe that $(0,\beta_1,0,\beta_2) \in H$ if and only if $z' | \beta_2$, i.e. $\beta_2 = z' \beta_2'$ for some $\beta_2' \in \Z/(N'/z')\Z$, and $\beta_1$ is of the form 
	$\beta_1 = -\beta_2' u z + \beta_1' N/x$ for some $\beta_1' \in \Z/x\Z$.
	
	Using Proposition \ref{WeylVect}, we see that the first component of the corresponding Weyl vector is $\rho_1 = \frac 1 {24} \sum_{\beta_2 \in \Z/N'\Z} \sum_{\beta_1 \in \Z/N\Z} c_{(0,\beta_1,0,\beta_2)}$, which is essentially the number of elements of $H$ of the form $(0,\beta_1,0,\beta_2)$. By the counting argument above, there are $Nx'/z =N'x/z'$ many of them, i.e., $\rho_1 = (1/24) N'x/z'$.
	Similarly, $\rho_2$ is given by the number of elements of $H$ of the form $(0,\beta_1,\beta_2,0)$, which is given by $xz'$ by the above. Hence, $\rho_2 = \frac{xz'}{24N'}$.

	By Theorem \ref{corBP2g}, the lift of $v^H$ factors as $\Psi(\tau_1, \tau_2) = C \psi_1(\tau_1) \psi_2(\tau_2)$ for some constant $C \in \C^\times$. 
	Moreover, we obtain using the calculations above that
	\begin{align*}
		\psi_1(\tau) & =  q^{\rho_1} \prod_{\lambda = 1}^\infty \prod_{\beta \in \Z/N\Z} \left(1- q^{\lambda} \zeta_N^\beta \right)^{c_{(0,\beta, \lambda,0)}} \\
		& =  q^{\rho_1} \prod_{ \lambda' = 1}^\infty \prod_{\beta' \in \Z/x\Z} \left( 1- q^{N'\lambda'/z'} \zeta_N^{\beta' N/x} \right) \\
		& =  q^{\rho_1} \prod_{\lambda' \in \Z} \left( 1 - q^{N'x/z' \cdot \lambda'} \right) \\
		& =  \eta\left( \frac{N'x}{z'} \tau \right),
	\end{align*}
	where we wrote $\lambda = \frac{N' \lambda'}{z'}$ and $\beta = \frac{N\beta'}{x}$ in the first step. 
	For the second component we find that
	\begin{align*}
		\psi_2(\tau) & =  q^{\rho_2} \prod_{\lambda = 1}^\infty \prod_{\beta \in \Z/N\Z} (1-q^{\lambda/N'} \zeta_N^\beta)^{c_{(0,\beta,0,\lambda)}} \\
		& =  q^{\rho_2} \prod_{\lambda' = 1}^\infty \prod_{\beta' \in \Z/x\Z} \left( 1- \left( q^{z'/N'} \zeta_N^{-uz} \right)^{\lambda'} \zeta_{x}^{\beta'} \right) \\
		& =  q^{\rho_2} \prod_{\lambda' = 1}^\infty \left( 1- \eb\left( \frac{x z' \tau}{N'} - \frac{u x z}{N} \right)^{\lambda'} \right) \\
		& =  \eb \left(\frac{uxz}{24N} \right) \eta \left(\frac{xz' \tau}{N'} - \frac{uxz}{N} \right),
	\end{align*}
	where $\lambda = z' \lambda'$ and $\beta = - \lambda' uz + \frac{N \beta'}{x}$.
	The calculation in the second case is done similarly.
\end{proof}

If we restrict ourselves to the case $N'=1$, Corollary \ref{corBasis} gives a complete description of all holomorphic modular forms of weight 0 and representation $\varrho_{L_{N,1}}$. We can thus describe all modular forms with respect to $\Gamma_{L_{N,1}}$ arising as Borcherds lifts of invariant vectors in the following way.
\begin{Korollar}
	\label{corN1}
	Let \(F = \sum_{d|N} \alpha_d v^{H_{d, 0, N/d}} \) be a holomorphic modular form of weight 0 and representation \(\varrho_{L_{N,1}}\) with integer coefficients. Then the lift $\Psi$ of $F$ is a modular form of weight $\frac 1 2 \sum_{d|N} \alpha_d$ with respect to $\Gamma_0(N)^2$ and has a product expansion in the cusp $(\infty, \infty)$ of the form
	$$\Psi(\tau_1, \tau_2) =C  \prod_{d | N} \eta(d \tau_1)^{\alpha_d} \cdot \prod_{d | N} \eta(d \tau_2)^{\alpha_d}$$
	for some constant $C \in \C^\times$. 
\end{Korollar}
\begin{proof}
	By Proposition \ref{exLiftBasis}, the vector $v^{H_{d,0,N/d}}$ lifts to $C \cdot \eta(d\tau_1) \cdot \eta(d\tau_2)$.
	Hence, the lift $\Psi$ of $F$ can be written as claimed.
	As an $\eta$-quotient $\prod_{d | N} \eta(d \tau)^{\alpha_d}$ is a modular form for $\Gamma_0(N)$ (with respect to some multiplier system), $\Psi$ is already a modular form for $\Gamma_0(N)^2$.
\end{proof}

Note that $\Gamma_{L_{N,1}}$ is a proper subgroup of $\Gamma_0(N)^2$. Due to the fact that the lift $\Psi$ factors as a product of functions in the two variables, we get modularity with respect to the bigger group $\Gamma_0(N)^2$ here.

We briefly analyse the character associated with $\Psi$ in the case $N' = 1$. In particular, we want to know, when the character becomes trivial.
\begin{Korollar}
	Let $F = \sum_{d|N} \alpha_d v^{H_{d,0,N/d}}$ be a holomorphic modular form of weight 0  and representation $\varrho_{L_{N,1}}$ such that the coefficients of \(F\) satisfy 
	\begin{equation*}
		\label{cond1}
		\sum_{d|N} d \alpha_d \equiv 0 \mod 24 \qquad 
		\text{and} \qquad \sum_{d|N} \frac N d \alpha_d \equiv 0 \mod 24,
	\end{equation*}
	$k = \frac 1 2 \sum_{d|N} \alpha_d$ is an even integer, and $s = \prod_{d |N} d^{\alpha_d}$ is a square in $\Q$. Then, the Borcherds product $\Psi$ associated with $F$ is a weight $k$ modular form for $\Gamma_0(N)^2$ with trivial character. 
\end{Korollar}
\begin{proof}
	By \cite[Theorem 1.64]{Ono}, the function $\psi(\tau) = \prod_{d|N} \eta(d \tau)^{\alpha_d}$ is a modular form with respect to $\Gamma_0(N)$ with trivial character. Hence, $\Psi$ is a modular form for $\Gamma_0(N)^2$ with trivial character.
\end{proof}
\begin{Bemerkung}
	If under the assumptions of the corollary we also have that the weight $k = \frac 1 2 \sum_{d|N} \alpha_d = 0$, then we obtain that $\Psi$ is indeed a function on the surface $X_0(N)^2$ with divisor that is supported on the boundary.
\end{Bemerkung}

\begin{Beispiel}
	We can also compute the lift of the characteristic function of the isotropic subgroup $ H = H_{(x,y,z),(x,-y,z)}^{ (-ux,u^{-1}y), (0,u^{-1}z)} $ of $D_{N,N}$ from Example \ref{exy-y}. In the same way as before, the lift of $v^H$ is given by
	$$ \Psi(\tau_1, \tau_2) = C \eta\left( \frac{Nx\tau_1 - uy}{z} \right) \eta\left( \frac{xz(\tau_2 +u)}{N}\right) $$
	for some constant $C \in \C^\times$.
\end{Beispiel}
%

We want to apply the results above to lift non-trivial linear relations among characteristic functions of self-dual isotropic subgroups of \(D_{p,p}\) to obtain identities between the $\eta$-function.
\begin{Korollar}
	\label{thmEta}
	Let $p$ be a prime number. We have the following identity between $\eta$-functions
	$$ \prod_{a = 1}^{p-1} \eta \left( \tau + \frac{a}{p} \right) = \eb\left(\frac{p-1}{48}\right) \frac{\eta(p\tau)^{p+1}}{\eta(\tau) \eta(p^2\tau)}.$$
\end{Korollar}

This identity is well-known for \(p=2\), see for example \cite{Eta}, but for $p \geq 3$ the result seems to be new.

\begin{proof}
	Recall from the proof of Proposition \ref{propInvP} the shorthand notation for the self-dual isotropic subgroups of $D_{p,p}$:
	$$H^{(1)} = H_{1,0,p} \oplus H_{1, 0, p}, \hspace{1.7mm} H^{(2)} = H_{1, 0, p} \oplus H_{p,0,1}, \hspace{1.7mm} H^{(3)} = H_{p,0,1} \oplus H_{1, 0, p}, \hspace{1.7mm} H^{(4)} = H_{p, 0, 1} \oplus H_{p,0,1},  $$
	$$ H^{(5)}_u = H_{(1,0,1),(1,0,1)}^{(u,0), (0, -u^{-1})}, \qquad \text{ and } \qquad  H^{(6)}_u = H_{(1,0,1),(1,0,1)}^{(0, u), (-u^{-1},0)},$$
	for $u \in (\Z/p\Z)^\times$. 
	We write $H^{(1)} = H_{(1,0,p),(1,0,p)}^{(1,0),(0,-p)} = H_{(1,0,p),(1,0,p)}^{(1,0),(0,0)}$, and obtain by Proposition \ref{exLiftBasis} that $v^{H^{(1)}}$ lifts to $\Psi^{(1)} (\tau_1, \tau_2) = C_1 \eta(\tau_1) \eta(\tau_2 - 1) = C_2 \eta(\tau_1) \eta(\tau_2)$ for some constant $C_1 \in \C^\times$ and $C_2 = \zeta_{24}^{-1} C_1$. Similarly, we calculate the lifts of $v^{H^{(i)}}$ for $i = 2,3,4$.
	By Proposition \ref{propInvP}, the characteristic functions satisfy the linear relation
	$$ v^{H^{(1)}}-v^{H^{(2)}} - v^{H^{(3)}} + v^{H^{(4)}} + \sum_{u = 1}^{p-1} \left( -v^{H^{(5)}_u} + v^{H^{(6)}_u} \right) = 0.$$
	We lift this relation using the calculations of Proposition \ref{exLiftBasis} and obtain
	$$
		\qquad \qquad \frac{\eta(\tau_1) \eta(\tau_2) \eta(p^2 \tau_1) \eta(\tau_2) \prod_u (\eta(\tau_1 - u/p) \eta( \tau_2))}{\eta(p \tau_1) \eta(\tau_2/p) \eta(p \tau_1) \eta(p\tau_2) \prod_u (\eta(p \tau_1) \eta(\tau_2/p - u/p))}  = C  $$
	$$	\Leftrightarrow \qquad \frac{\eta(\tau_1) \eta(p^2 \tau_1)  \prod_u \eta(\tau_1 -u/p)}{\eta(p \tau_1)^{p+1}}  =  C \cdot \frac{\eta(\tau_2/p) \eta(p \tau_2) \prod_u \eta(\tau_2/p - u/p)}{\eta(\tau_2)^{p+1}}
	$$
	for some constant $C \in \C^\times$ (which can be seen to equal 1 after taking $\tau_2 = p\tau_1$). 
	Since the right hand side does not depend on \(\tau_1\), we have (after setting $a = -u$)
	$$\prod_{a = 1}^{p-1} \eta \left( \tau + \frac{a}{p} \right) = C' \frac{\eta(p\tau)^{p+1}}{\eta(\tau) \eta(p^2\tau)}$$
	for some further constant \(C' \in \C^\times\). It remains to determine the value of the constant. We compare the coefficients in the \(q\)-expansion of the two terms. We have
	\[ \prod_{a = 1}^{p-1} \eta \left( \tau + \frac{a}{p} \right) = \prod_{a = 1}^{p-1} \left( \eb\left(\frac{a}{24p}\right) q^{1/24}  \prod_{n=1}
	^\infty \left( 1 - \eb\left(\frac{na}{p}\right) q^n \right) \right), \]
	so the \(q^{(p-1)/24}\) term has the coefficient $\eb\left( \frac{\sum_{a=1}^{p-1} a}{24p}\right) = \eb\left(\frac{p-1}{48}\right).$
	Since the \(\eta\)-quotient on the right hand side has the coefficient 1 for its \(q^{(p-1)/24}\) term, we obtain 
	$C' = \eb\left(\frac{p-1}{48}\right).$
\end{proof}

Note that the lifts of the relations among characteristic functions in $\mathcal{M}_{0, D_{N,p}}$ in Proposition \ref{propInvP} for general $N$ all give rise to the identity between $\eta$-quotients in Corollary \ref{thmEta}.

\section*{Acknowledgements}
The results in this paper first appeared in my Master's thesis \cite{Bieker} supervised by Jan Hendrik Bruinier and Paul Kiefer.
I thank Jan Hendrik Bruinier for suggesting the topic and for the continuous support during the preparation of my thesis. I thank Paul Kiefer for the many helpful discussions and remarks. I thank Timo Richarz and Shaul Zemel for their comments and suggestions.
I also thank Nils Scheithauer for his interest in this work and in particular for sharing the content of \cite{Mueller21} with me. 

I thank the referee for their numerous helpful remarks that lead to an enormous improvement of the exposition.
\newline

\noindent Funding: This work was partially funded by the Deutsche Forschungsgemeinschaft (DFG, German Research Foundation) TRR 326 \textit{Geometry and Arithmetic of Uniformized Structures}, project number 444845124.

\printbibliography

\end{document}